\documentclass[12pt,a4paper]{article}

\usepackage{amssymb}

\usepackage{amsmath}
\usepackage{amsthm}

\usepackage[color,all,cmtip]{xy} 

\usepackage{tikz}
\usetikzlibrary{matrix}

\usetikzlibrary{arrows}

\usepackage{url}
\urldef{\mailsa}\path|arojas|
\urldef{\mailsc}\path|@dc.uba.ar|

\newcommand{\ol}{\overline}

\newcommand{\A}{\mathbb A}
\newcommand{\C}{\mathbb C}
\newcommand{\N}{\mathbb N}
\newcommand{\R}{\mathbb R}
\newcommand{\Z}{\mathbb Z}
\newcommand{\G}{\mathbb G}

\newcommand{\0}{\textbf{0}}
\newcommand{\klk}{,\ldots,}
\newcommand{\cH}{\mathcal H}
\newcommand{\cN}{\mathcal N}
\newcommand{\cM}{{\mathcal{M}}}
\newcommand{\cU}{{\mathcal{U}}}
\newcommand{\cD}{{\mathcal{D}}}
\newcommand{\cS}{{\mathcal{S}}}
\newcommand{\cO}{{\mathcal{O}}}
\newcommand{\m}{\mathfrak{m}}

\newtheorem{claim}{Claim}

\newtheorem{theorem}{Theorem}
\newtheorem{corollary}[theorem]{Corollary}
\newtheorem{definition}[theorem]{Definition}
\newtheorem{example}[theorem]{Example}
\newtheorem{lemma}[theorem]{Lemma}
\newtheorem{proposition}[theorem]{Proposition}
\newtheorem{remark}[theorem]{Remark}

\usepackage{hyperref}
\hypersetup{
     colorlinks   = true,
     linkcolor  = black,
     citecolor  = black
}

\newcommand{\boldface}[1]{{ \textbf{#1} }}

\begin{document}

\title{Quiz Games as a model for Information Hiding}
\author{Bernd Bank$^{1}$, Joos Heintz$^{2,3}$, Guillermo Matera$^{3,4}$,\\
Jos\'e Luis Monta\~na$^{5}$,
Luis M. Pardo$^{5}$, Andr\'es Rojas Paredes$^{2}$}

\footnotetext[0]{Research was partially supported by the following
Spanish and Argentinean grants: MTM2010-16051, MTM2014-55262-P, PIP
CONICET 11220130100598, PIO CONICET-UNGS 14420140100027, UNGS
30/3084 and UBACyT 20020130100433BA.}

\footnotetext[1]{Humboldt-Universit\"at zu Berlin, Institut f\"ur
Mathematik, 10099 Berlin, Germany. {\it Email address}: {\tt
bank@mathematik.hu-berlin.de}}

\footnotetext[2]{Departamento de Computaci\'on, Facultad de Ciencias
Exactas y Naturales, Universidad de Buenos Aires, Ciudad Univ., Pab.
I, 1428 Buenos Aires, Argentina. {\it Email address}: {\tt
\{joos,arojas\}@dc.uba.ar}}

\footnotetext[3]{National Council of Science and Technology
(CONICET), Ar\-gentina}

\footnotetext[4]{Instituto del Desarrollo Humano, Universidad
Nacional General Sarmiento, J. M. Guti\'errez 1150 (B1613GSX) Los
Polvorines, Provincia de Buenos Aires, Argentina. {\it Email
address}: {\tt gmatera@ungs.edu.ar}}

\footnotetext[5]{Departamento de Matem\'aticas, Estad\'{\i}stica y
Computaci\'on, Facultad de Ciencias, Universidad de Cantabria, 39071
Santander, Spain. {\it Email address}: {\tt joseluis.montana@unican.es, luis.m.pardo@gmail.com}}

%

\maketitle
%
%
\begin{abstract}
We present a general computation model inspired in the notion of
information hiding in software engineering. This model has the form
of a game which we call \emph{quiz game}. It allows in a uniform way
to prove exponential lower bounds for several complexity problems.
\end{abstract}
%
%
\section{Introduction}
We present a general computation model inspired in the notion of
information hiding in software engineering. This model has the form
of a game which we call \emph{quiz game}. It consists of a one round
two-party protocol between two agents, namely a \emph{quizmaster}
with limited and a {\emph{player} with unlimited computational
power. We suppose that the quizmaster is honest and able to answer
the player's questions. Using this model we are able to prove
exponential lower bounds for several complexity problems in a
uniform way, for example for the continuous interpolation of
multivariate polynomials of given circuit complexity (Theorem
\ref{theorem lower bound}). It is also possible to exhibit sequences
of families of multivariate polynomials which are easy to evaluate
such that the continuous interpolation of these polynomials, or
their derivatives or their indefinite integrals require an amount of
arithmetic operations which is exponential in their circuit
complexity (Theorem \ref{less technical theorem}). On the other
hand, we represent neural networks with polynomial activation
functions in our model and show that there is no continuous
algorithm able to learn relatively simple neural networks exactly
(Theorem \ref{neural-networks:teor}). Finally, we exhibit infinite
families of first--order formulae over $\mathbb{C}$ which can be
encoded in polynomial time and determine classes of univariate
parameterized elimination polynomials such that any representation
of these classes is of exponential size (Theorems
\ref{elimination-exponential-size:teor} and
\ref{characteristic-polynomial-exponential-size:teor}).

Ad hoc variants of the method we use and partial results already
appeared elsewhere (\cite{GiHeMaSo11}, \cite{HeKuRo}). What is
really new is the general framework which we develop to approach
these complexity results in order to prove (and generalize) them in
a uniform way.

The quiz games which constitute the core of our model admit two
``protocols'', an  ``exact'' and an ``approximative'' one. The exact
protocol aims to represent symbolic procedures for solving
parametric families of elimination problems and is first discussed
in the context of robust arithmetic circuits and neural networks.
The approximative protocol is able to deal with information of
approximative nature. It is motivated by the notion of an
approximative parameter instance which encodes a polynomial with
respect to an abstract data type. The main outcome is that there
exists an approximative parameter instance encoding a given
polynomial if and only if that polynomial belongs to the closure of
the corresponding abstract data type with respect to the Euclidean
topology.

The idea behind this computational model is to restrict the
information which quizmaster and player may interchange. This
reflects the concept of information hiding in software engineering
aimed to control and reduce the design complexity of a computer
program.

In the most simple case the notion of an exact quiz game protocol
may be explained roughly as follows. Suppose that there is given a
continuous data structure carrier together  with an abstraction
function which encodes a parameterized family of polynomials. The
quizmaster chooses from the data structure carrier a parameter which
encodes a specific polynomial and hides it to the player. The player
asks to the quizmaster questions about the hidden polynomial, whose
answers constitute a vector of complex values which depend only on
the polynomial itself and are independent of the hidden parameter.
The quizmaster sends this vector to the player and the player
computes a representation of the polynomial in an alternative data
structure carrier. Finally, the quizmaster tests whether this
alternative representation encodes the hidden polynomial. Observe
that polynomial interpolation is a typical situation that can be
formulated in such a way.

The paper constitutes a mixture between ideas and concepts coming
from software engineering, algebraic complexity theory and algebraic
geometry. A fundamental tool is an algebraic characterization of the
total maps whose graphs are first--order definable over $\mathbb{C}$
and continuous with respect to the Euclidean topology. We call these
maps \emph{constructible and geometrically robust} (see Theorem
\ref{theo continuos}).
%
%
\section{Concepts and tools from algebraic geometry}
In this section, we use freely standard notions and notations from
commutative algebra and algebraic geometry. These can be found for
example in \cite{Eisenbud95}, \cite{Kunz85}, \cite{Mumford88}  and
\cite{Shafa94}. In Section \ref{subsec: geom robust const maps} we
introduce the notions and definitions which constitute the
fundamental tool for our algorithmic models. Most of these notions
and their definitions are taken from \cite{GiHeMaSo11} and
\cite{HeKuRo}.
%
%
\subsection{Basic notions and notations}
Let $k$ be a fixed algebraically closed field of characteristic
zero. For any $n\in\N$, we denote by $\A^n(k)$ the $n$--dimensional
affine space $k^n$ equipped with its Zariski topology. For $k=\C$,
we consider the complex $n$--dimensional affine space
$\A^n:=\A^n(\C)$, equipped with its respective Zariski and Euclidean
topologies.

Let $X_1,\dots,X_n$ be indeterminates over $k$ and let
$X:=(X_1,\dots,X_n)$. We denote by $k[X]:=k[X_1,\dots,X_n]$ the ring
of polynomials in the variables $X$ with coefficients in $k$.

Let $V$ be a closed affine subvariety of $\A^n(k)$. We denote by
$I(V):=\{ f \in k[X]: f(x)=0 \text{\ for\ any\ } x \in V\}$ the
ideal of definition of $V$ in $k[X]$ and by $k[V]:=\{
\phi:V\rightarrow k: \text{\ there\ exists\ } f\in k[X] \text{\
with\ } \phi(x)=f(x) \text{\ for\ any\ } x \in V \}$ its coordinate
ring. The elements of $k[V]$ are called coordinate functions of $V$.
Observe that $k[V]$ is isomorphic to the quotient $k$--algebra
$k[X]/{I(V)}$. If $V$ is irreducible, then $k[V]$ has no zero
divisors, and we denote by $k(V)$ the field formed by the rational
functions of $V$ with maximal domain ($k(V)$ is called the rational
function field of $V$). Observe that $k(V)$ is isomorphic to the
fraction field of the integral domain $k[V]$.

Let $V$ and $W$ be closed affine subsets of $\A^n(k)$ and $\A^m(k)$,
respectively, and let $\Phi:V\rightarrow W$ be a (total) map. We
call $\Phi$ a \emph{morphism} from the affine variety $V$ to $W$ if
there exist polynomials $f_1,\dots,f_m\in k[X]$ such that
$\Phi(x)=(f_1(x),\dots,f_m(x))$ holds for any $x\in V$.

Let $V$ be irreducible, let $U$ be a nonempty Zariski open subset of
$V$ and let $\Phi:V\dashrightarrow W$ be a partial map with domain
$U$. Let $\Phi_1,\dots,\Phi_m$ be the components of $\Phi$. We call
$\Phi$ a \emph{rational map} from $V$ to $W$ if
$\Phi_1,\dots,\Phi_m$ are the restrictions to $U$ of suitable
rational functions of $V$. Observe that our definition of a rational
map differs slightly from the usual one in algebraic geometry, since
we do not require that the domain $U$ of $\Phi$ is maximal. Hence in
the case $m:=1$ our concepts of rational function and rational map
do not coincide. However we will not stick on this point and simply
speak about rational functions when $m=1$.
\begin{example}
Let $f:\A^2\dashrightarrow\A^1$ be the map defined as
$f(X_1,X_2):=\frac{X_1+X_2}{X_1-X_2}$ on
$U:=\A^2\setminus\{X_1^2-X_2^2=0\}$. Then $f$ may be considered as a
rational map in the sense above, but it is not a rational function
on $\A^2$.
\end{example}
%
%
\subsection{Geometrically robust constructible maps}
\label{subsec: geom robust const maps}
Let $\mathcal{M}$ be a subset of some affine space $\A^n(k)$ and,
for a given non-negative integer $m$, let
$\Phi:\mathcal{M}\dasharrow\A^m(k)$ be a partial map.

\begin{definition}
We call the set $\mathcal{M}$ constructible if $\mathcal{M}$ is
definable by a Boolean combination of polynomial equations from
$k[X]$, namely as a finite union of sets of solutions of equalities
and inequalities defined by elements of $k[X]$.
\end{definition}

Since the elementary theory of algebraically closed fields of
characteristic zero admits quantifier elimination (see, e.g.,
\cite{FrJa05}), constructible and definable sets in the first--order
logic over $k$ are exactly the same.

For a constructible subset $\mathcal{M}$ of $\A^n(k)$, we denote its
Zariski closure by $\overline{\mathcal{M}}$. For $k=\C$, the Zariski
closure a constructible subset $\mathcal{M}$ of $\A^n$ coincides
with its Euclidean closure (see, e.g., \cite[Chapter I, Section 10,
Corollary 1]{Mumford88}). Hence the notation $\ol{\mathcal{M}}$ for
the closure of $\mathcal{M}$ with respect to both topologies is
unambiguous.

A constructible subset $\mathcal{M}$ of $\A^n(k)$ is called
\emph{irreducible} if it cannot be written as a nontrivial union of
two subsets of $\mathcal{M}$ which are closed with respect to the
Zariski topology of $\mathcal{M}$. Each constructible subset
$\mathcal{M}$ of $\A^n(k)$ has a unique irredundant irreducible
decomposition as a finite union of irreducible, constructible
subsets of $\mathcal{M}$, which are closed in $\mathcal{M}$. These
subsets are called the \emph{irreducible components} of
$\mathcal{M}$.

\begin{definition}
We call the partial map $\Phi$ constructible if the graph of $\Phi$
is constructible as a subset of the affine space
$\A^n(k)\times\A^m(k)$.
\end{definition}

We say that the constructible map $\Phi$ is \emph{polynomial} if
$\Phi$ is the restriction to $\mathcal{M}$ of a morphism of affine
varieties $\A^n(k)\rightarrow \A^m(k)$. A polynomial map
$\Phi:\mathcal{M}\to\A^m(k)$ is everywhere defined on $\mathcal{M}$
and hence total. The constructible map $\Phi$ is \emph{rational} if
the intersection of its domain with any irreducible component
$\mathcal{N}$ of $\mathcal{M}$ is a nonempty open subset
$\mathcal{U}$ of the closure $\overline{\mathcal{N}}$ of
$\mathcal{N}$ in the Zariski topology of $\A^n(k)$ and the
restriction $\Phi|_{\mathcal{U}}$ of $\Phi$ to $\mathcal{U}$ is a
rational map of $\overline{\mathcal{N}}$. In this case $\mathcal{U}$
is a Zariski dense subset of $\mathcal{N}$.

\begin{remark}
\label{remark piecewise rational} A partial map
$\Phi:\mathcal{M}\dasharrow\A^m(k)$ is constructible if and only if
it is piecewise rational. If $\Phi$ is a total constructible map,
then there exists an open subset $\mathcal{U}$ of $\mathcal{M}$ with
nonempty intersection with any irreducible component of
$\mathcal{M}$ such that $\Phi|_\mathcal{U}$ is a rational map of
$\mathcal{M}$.
\end{remark}

The first statement follows from quantifier elimination, whereas the
second follows from \cite[Lemma 1]{GiHeMaSo11}.

Fix for the moment an irreducible constructible subset $\mathcal{M}$
of the affine space $\A^n(k)$ and a total constructible map
$\Phi:\mathcal{M}\rightarrow\A^m(k)$ with components
$\Phi_1,\dots,\Phi_m$. Observe that the closure
$\overline{\mathcal{M}}$ of $\mathcal{M}$ is an irreducible closed
affine subvariety of $\A^n(k)$ and that we may interpret
$k(\overline{\mathcal{M}})$ as a $k[\overline{\mathcal{M}}]$--module
(or $k[\overline{\mathcal{M}}]$--algebra). Fix now an arbitrary
point $x$ of $\overline{\mathcal{M}}$. By $\mathfrak{M}_x$ we denote
the maximal ideal of coordinate functions of
$\overline{\mathcal{M}}$ which vanish at the point $x$. By
$k[\overline{\mathcal{M}}]_{\mathfrak{M}_x}$ we denote the local
$k$--algebra of the variety $\overline{\mathcal{M}}$ at the point
$x$, i.e., the localization of $k[\overline{\mathcal{M}}]$ at the
maximal ideal $\mathfrak{M}_x$.

Following Remark \ref{remark piecewise rational} we may interpret
$\Phi_1,\dots,\Phi_m$ as rational functions of the irreducible
variety $\overline{\mathcal{M}}$ and therefore as elements of
$k(\overline{\mathcal{M}})$. Thus
$k[\overline{\mathcal{M}}]_{\mathfrak{M}_x}[\Phi_1,\dots,\Phi_m]$ is
a  $k$--subalgebra of $k(\overline{\mathcal{M}})$ which contains
$k[\overline{\mathcal{M}}]_{\mathfrak{M}_x}$.

\begin{definition}
\label{geom robust def}
Let $\mathcal{M}$ be a constructible subset of a suitable affine
space over $k$ and let $\Phi:\mathcal{M}\rightarrow\A^m(k)$ be a
total constructible map with components $\Phi_1,\dots,\Phi_m$. If
$\mathcal{M}$ is irreducible, then we call $\Phi$ geometrically
robust if for any point $x\in\mathcal{M}$ the following two
conditions are satisfied:
\begin{itemize}
    \item[$(i)$] $k[\overline{\mathcal{M}}]_{\mathfrak{M}_x}[\Phi_1,\dots,\Phi_m]$
    is a finite
    $k[\overline{\mathcal{M}}]_{\mathfrak{M}_x}$--module,
    \item[$(ii)$] $k[\overline{\mathcal{M}}]_{\mathfrak{M}_x}[\Phi_1,\dots,\Phi_m]$
    is a local $k[\overline{\mathcal{M}}]_{\mathfrak{M}_x}$--algebra whose maximal
    ideal is generated by $\mathfrak{M}_x$ and
    $\Phi_1-\Phi_1(x),\dots,\Phi_m-\Phi_m(x)$.
\end{itemize}
In the general case, let $\mathcal{N}_1,\dots,\mathcal{N}_s$ be the
irreducible components of $\mathcal{M}$. We call then $\Phi$
geometrically robust if the restrictions
$\Phi|_{\mathcal{N}_1},\dots,\Phi|_{\mathcal{N}_s}$ are
geometrically robust in the above sense.
\end{definition}

If $\Phi$ is a geometrically robust constructible map, we call
$\mathcal{M}$ the \emph{domain of definition} of $\Phi$. The
following statements  characterize the unique properties of
geometrically robust constructible maps which will be relevant in
the sequel.
\begin{theorem}
\label{theorem geometrically robust functions} ${}$
\begin{itemize}
    \item[$(i)$] The restriction of a geometrically robust constructible
    map to a constructible subset of its domain is geometrically robust.
    \item[$(ii)$] Compositions and cartesian products of geometrically
    robust con\-struc\-ti\-ble maps are geometrically robust.
    \item[$(iii)$] A geometrically robust constructible map
    defined on a normal (e.g., smooth) variety is a polynomial map.
    \item[$(iv)$] Polynomial maps are geometrically robust and constructible
    and the geometrically robust constructible functions of a constructible
    domain of definition form a $k$--algebra.
\end{itemize}
\end{theorem}

For a proof of the statements $(i)$ and $(iii)$ we refer to
\cite[Theorem 17 and Corollary 12]{GiHeMaSo11}, respectively. The
statements $(ii)$ and $(iv)$ are immediate consequences of
Definition \ref{geom robust def}.

With some extra effort one can also show that geometrically robust
constructible maps are continuous with respect to the Zariski
topologies of their domain and range spaces. We shall prove this
only in case $k:=\C$ (see Lemma \ref{lemma Phi}$(ii)$ in Appendix
\ref{section: futher facts on geom rob maps}). Furthermore, we have
the following result (for a proof, see \cite[Theorem 4]{HeKuRo}).
\begin{theorem}
\label{theo continuos} Let $\mathcal{M}$ be a constructible subset
of $\A^n$ and let $\Phi:\mathcal{M}\rightarrow\A^m$ be a
constructible total map. Then $\Phi$ is geometrically robust if and
only if $\Phi$ is continuous with respect to the Euclidean
topologies of $\mathcal{M}$ and $\A^n$.
\end{theorem}

Theorem \ref{theo continuos} gives a topological characterization of
the notion of geometrically robust constructible maps over $\C$.
This notion represents the real motivation of Definition \ref{geom
robust def}, both from a geometric as well as from an algorithmic
point of view.

In Appendix \ref{section: futher facts on geom rob maps} we
establish further facts on geometrically robust constructible maps
which will be needed in the sequel.
%
%
\section{The computation model}
In this section we present our computation model. It will be
expressed in terms of framed abstract data type carriers, framed
data structures and abstraction functions. These notions are first
informally discussed in the context of robust arithmetic circuits
and neural networks, and then in a general setting. Then we
introduce our model, which has the form of a game ---a quiz game---
and aims to represent the notion of information hiding in software
engineering. We shall present two ``protocols'' of a quiz game: an
``exact'' and an ``approximative'' one. In this section we consider
only the exact protocol, while the somewhat more subtle
approximative protocol is discussed in the next section.
%
%
\subsection{Robust arithmetic circuits}
\label{subsec: robust arithmetic circuits} Let us fix natural
numbers $r$ and $n$, indeterminates $X_1,\dots,X_n$ and a non--empty
constructible subset $\mathcal{M}$ of $\A^r$. A \emph{robust
arithmetic circuit} (with parameter domain $\mathcal{M}$ and inputs
$X_1,\dots,X_n$) is a labeled directed acyclic graph (labeled DAG)
$\beta$ satisfying the following conditions: each node of indegree
zero is labeled by a robust constructible function (called parameter
of $\beta$) with domain of definition $\mathcal{M}$ or by a variable
$X_1,\dots,X_n$ (called inputs of $\beta$). All other nodes of
$\beta$ have indegree two and are called \emph{internal}. They are
labeled by the arithmetic operations addition, subtraction or
multiplication. Moreover, exactly one node of $\beta$ becomes
labeled as output. We call the number of nodes the \emph{size} of
$\beta$.

We consider $\beta$ as a syntactic object which we think equipped
with the following semantics. There exists a canonical evaluation
procedure of $\beta$ assigning to each node a geometrically robust
constructible function with domain of definition
$\mathcal{M}\times\A^n$ which, in case of a parameter node, may also
be interpreted as a geometrically robust function with domain of
definition $\mathcal{M}$. In either situation, we call such a
function an \emph{intermediate result} of $\beta$. The intermediate
result associated with the output node will be called the
\emph{final result} of $\beta$. We refer to $\mathcal{M}$ as the
\emph{parameter domain} of $\beta$.

By definition, $\beta$ does not contain divisions involving the
inputs. Divisions may appear, only implicitly, in the construction
of the parameters of $\beta$. In this sense $\beta$ is
\emph{essentially division--free}. Therefore all intermediate
results of $\beta$ are polynomials in $X_1,\dots,X_n$ over the
$\C$--algebra of geometrically robust constructible functions with
domain of definition $\mathcal{M}$.

We may consider $\beta$ as a program which solves the problem to
evaluate, for each $u\in\mathcal{M}$, the polynomial function
$\A^n\rightarrow\A^1$ which we obtain by specializing to the point
$u$ the first argument of the geometrically robust constructible
function $\mathcal{M}\times \A^n\rightarrow \A^1$ given by the final
result of $\beta$. In this sense $\beta$ defines an
\emph{abstraction function} $\theta$ which assigns to each parameter
instance $u\in\mathcal{M}$ a polynomial function from $\A^n$ into
$\A^1$. These polynomial functions constitute a (unary)
\emph{abstract data type carrier} $\mathcal{O}$ which is contained
in a finite dimensional $\C$--vector subspace of $\C[X_1,\dots,X_n]$
and forms there a constructible subset. We call therefore this
carrier \emph{framed}. The abstraction function
$\theta:\mathcal{M}\rightarrow\mathcal{O}$ is geometrically robust
and constructible. The constructible parameter domain $\mathcal{M}$
constitutes a \emph{framed data structure} which represents the
elements of $\mathcal{O}$ by means of the geometrically robust
constructible map $\theta$.

Suppose that $\beta$ contains $K$ parameter nodes. The corresponding
parameters of $\beta$ realize a geometrically robust constructible
map $\mu:\mathcal{M}\rightarrow\A^K$ with constructible image
$\mathcal{N}$. Since the circuit $\beta$ is essentially
division--free, there exists a polynomial map
$\omega:\mathcal{N}\rightarrow\mathcal{O}$ such that
$\theta=\omega\circ\mu$ holds.

In Section \ref{subsec: data structures} we are going to axiomatize
this situation and to define precisely what we mean by the up to now
informal notions of framed abstract data type carrier, framed data
structure and abstraction function.  For more details about robust
arithmetic circuits and their motivations we refer to \cite{HeKuRo}.

Particular instances of robust arithmetic circuits are those whose
parameter domain are affine spaces. In this case all parameters are
polynomials (see Theorem \ref{theorem geometrically robust
functions}$(iii)$). These circuits represent abstraction functions
which encode polynomial functions defined on affine spaces by means
of \emph{ordinary division--free arithmetic circuits} (see
\cite{HeKuRo} for this terminology).

Let $L,n \in\N$ and let $\mathcal{O}_{L,n}$ be the family of all
polynomials of $\C[X_1,\dots,X_n]$ which can be evaluated by an
ordinary division--free arithmetic circuit with at most $L$
\emph{essential} multiplications, involving the input variables
$X_1,\dots,X_n$ meanwhile $\C$--linear operations are free. Then
$\mathcal{O}_{L,n}$ is an abstract data type whose abstraction
function is represented by a robust arithmetic circuit with
parameter domain $\A^{(L+n+1)^2}$ (see \cite[Exercise
9.18]{Burgisser97}). We call this arithmetic circuit the
\emph{generic computation} of all $n$--variate polynomials which can
be evaluated with at most $L$ essential multiplications.
%
%
\subsection{Neural networks with polynomial activation functions}
\label{subsec: neural networks}
In this section we will freely use well--established terminology on
neural networks (see, e.g., \cite{Hertz91}, \cite{Haykin} or
\cite{Hagan}). Let be given a neural network architecture with  $n$
inputs $X_1,\dots,X_n$ and one output and let $r$ be the length of
the corresponding weight vector. We suppose that all activation
functions are given by univariate polynomials over $\R$. Observe
that each weight vector instance $w\in\R^r$ defines a neural network
of this architecture and thus a polynomial target function from
$\R^n$ into $\R$. The dependency weight vector instance--target
function is itself polynomial and can be extended to $\A^r$. The
complex target functions we obtain in this way constitute a (unary)
abstract data type carrier $\mathcal{O}$ which is contained in a
finite dimensional $\C$--vector subspace of $\C[X_1,\dots,X_n]$ and
forms there a constructible subset. In this sense, the abstract data
type carrier $\mathcal{O}$ is again framed. Thus we obtain for
$\mathcal{M}:=\A^r$ a surjective polynomial map
$\theta:\mathcal{M}\rightarrow\mathcal{O}$ which may be interpreted
as an abstraction function which represents the elements of
$\mathcal{O}$ by means of complex weight vector instances belonging
to $\mathcal{M}$. In this meaning $\mathcal{M}$ constitutes a framed
data structure. Since $\theta$ is polynomial, we may obviously write
it as a composition of a polynomial and a geometrically robust
constructible map. Hence we shall interpret again $\theta$ as an
abstraction function.
%
%
\subsection{Framed abstract data type carriers, framed data structures
and abstraction functions} \label{subsec: data structures}
Now we define in a general context the notions of framed abstract
data type carrier, framed data structure and abstraction function.
Let $(X_n)_{n\in\N}$ be a sequence of indeterminates over $\C$ and
let
\[
\mathcal{R} := \bigcup_{n\in\N} \C[X_1,\dots,X_n].
\]
A {\it framed (unary) abstract data type carrier of polynomials} is
a constructible subset of a finite--dimensional $\C$--vector space
contained in $\mathcal{R}$. In the following we shall only refer to
unary abstract data type carriers and omit the expression ``unary''.
Only at the end of this section we shall briefly mention $r$--ary
abstract data type carriers for arbitrary $r\in\N$.

By \emph{framed data structures} we refer to constructible subsets
of suitable affine ambient spaces over $\C$. For a framed data
structure $\mathcal{M}$, the {\it size} of $\mathcal{M}$ is the
dimension of its ambient space.

Let $\mathcal{O}$ be a framed abstract data type carrier of
polynomials, $\mathcal{M}$ and $\mathcal{N}$ framed data structures,
$\mu:\mathcal{M}\to \mathcal{N}$ a geometrically robust
constructible map and $\omega:\mathcal{N} \to \mathcal{O}$ a
polynomial map such that the geometrically robust constructible map
$\theta:=\omega \circ \mu$ sends $\mathcal{M}$ onto $\mathcal{O}$.
The situation may be depicted by the following commutative diagram:
\[
\xymatrix{
\cO     &  \\
\cM \ar[r]_{\mu}  \ar[u]^{\theta}    & \cN \ar[ul]_{\omega}}
\]
We call $\theta$ the {\it abstraction function} associated with
$\mu$ and $\omega$. The maximal size of $\mathcal{M}$ and
$\mathcal{N}$ is called the \emph{size} of $\theta$. Observe that
the topological closures of $\mathcal{O}$ and $\mathcal{N}$ are well
defined with respect to the Zariski and Euclidean topologies and
coincide. We denote them by $\ol{\mathcal{O}}$ and
$\ol{\mathcal{N}}$. The polynomial map $\omega$ sends
$\ol{\mathcal{N}}$ into $\ol{\mathcal{O}}$.

The idea behind our notion of abstraction function is the following.
The specification language we use to speak about $\mathcal{R}$
consists of constants from $\C$, the arithmetic operations addition,
subtraction and multiplication, and equality. Thus a framed abstract
data type carrier $\mathcal{O}$ of polynomials can always (not
necessarily efficiently) be described by a formula or a
division-free arithmetic circuit which depends on a suitable framed
data structure $\mathcal{N}$ of parameters. The coefficient-wise
representation of the elements of $\mathcal{O}$ defines  a
surjective polynomial map $\omega:\mathcal{N} \to \mathcal{O}$.

In the case of the representation of polynomials by essentially
division--free, robust arithmetic circuits described in Section
\ref{subsec: robust arithmetic circuits}, the size of $\theta$ is
evidently a lower bound for the DAG size of the circuit $\beta$.
This example shows that also in the general case the size of
$\theta$ is a reasonable measure for the representation complexity
of the elements of $\mathcal{O}$ by means of $\theta$.

We allow now that $\mathcal{N}$ becomes \emph{re--parameterized} by
a geometrically robust constructible map $\mu$ with domain of
definition $\mathcal{M}$ and with $\mu(\mathcal{M})\subset
\mathcal{N}$ such that the composite map $\theta=\omega \circ \mu$
sends $\mathcal{M}$ onto $\mathcal{O}$. In this sense, our notion of
abstraction function for framed abstract data type carriers of
polynomials is absolutely natural and contains as first class
citizens the representation of polynomial families by means of
robust arithmetic circuits.

All we have said before (and we shall say in the sequel) about
framed \emph{unary} abstract data type carriers may be applied,
mutatis mutandis, to the case where the ring $\mathcal{R}$ is
replaced by its cartesian product $\mathcal{R}^r$, for $r\in\N$. If
this occurs, we speak about \emph{framed $r$--ary abstract data type
carriers}.
%
%
\subsubsection{Identification sequences}
Let be given two framed abstract data type carriers $\mathcal{O}_1$
and $\mathcal{O}_2$, two framed data structures $\mathcal{N}_1$ and
$\mathcal{N}_2$ and two surjective polynomial maps
$\omega_1:\mathcal{N}_1\rightarrow\mathcal{O}_1$ and
$\omega_2:\mathcal{N}_2\rightarrow\mathcal{O}_2$. Then the number of
variables and the degrees of the polynomials occurring in
$\mathcal{O}_1$ and $\mathcal{O}_2$ are bounded. Therefore we may
assume without loss of generality that the elements of
$\mathcal{O}_1$ and $\mathcal{O}_2$ are polynomials of
$\C[X_1,\dots,X_n]$ of degree bounded by a fixed integer parameter
$\Delta \geq 2$. Let $L$ be the size of the framed data structure
$\mathcal{N}_1\times \mathcal{N}_2$ and suppose that there exist a
quantifier--free first--order formula over $\C$ which defines
$\mathcal{N}_1\times\mathcal{N}_2$ involving $K$ polynomial
equations of degree at most $\Delta$ in $L$ variables. Taking into
account that there are at most $\Delta^L$ such equations which are
linearly--independent over $\C$, we may assume without loss of
generality $K \leq \Delta^L$. Finally we assume that the degree of
the polynomials defining $\omega_1\times\omega_2:
\mathcal{N}_1\times\mathcal{N}_2 \rightarrow
\mathcal{O}_1\times\mathcal{O}_2$ is bounded by $\Delta$.

For $m\in\N$, we call $(\gamma_1,\dots,\gamma_m)\in(\A^n)^m$ an
\emph{identification sequence} for $\mathcal{O}_1\times
\mathcal{O}_2$ if the equalities
$f(\gamma_1)=g(\gamma_1),\dots,f(\gamma_m)=g(\gamma_m)$ imply $f=g$
for any $f\in\ol{\mathcal{O}_1}$, $g\in\ol{\mathcal{O}_2}$.

The next statement assures that there exist for $\mathcal{O}_1\times
\mathcal{O}_2$ many integer identification sequences of small length
$m:=4L+2$ and of small bit size $O(L \text{\ log\ }\Delta)$ which
may be chosen randomly. For the proof of this result, we refer to
\cite[Lemma 4 and Corollary 1]{CaGiHeMaPa03}.
\begin{proposition}
\label{propo M} Let notations and assumptions be as before. Let $M$
be a finite subset of $\A^1$. Suppose that the cardinality $\# M$ of
$M$ satisfies the estimate $\# M \geq \Delta^3
(1+L)^{\frac{1}{L}}(1+ K \Delta)$ and let be given an integer $m
\geq 4L+2$. Then there exist points $\gamma_1,\dots,\gamma_m$ of
$M^n$ such that $(\gamma_1,\dots,\gamma_m)$ forms an identification
sequence for $\mathcal{O}_1\times \mathcal{O}_2$.

Suppose that the points of the finite set $M^n$ are equidistributed.
Then the probability of finding by a random choice in $(M^n)^m$ an
identification sequence is at least $1-\frac{1}{\# M} \geq
\frac{1}{2}$.
\end{proposition}

Suppose that for any $(v_1,v_2) \in \ol{\mathcal{N}_1} \times
\ol{\mathcal{N}_2}$ and any $\xi\in\A^n$ we are able to evaluate
$\omega_1(v_1)(\xi)$ and $\omega_2(v_2)(\xi)$ efficiently, i.e.,
using a number of arithmetic operations in $\C$ which is polynomial
in $L$ (this occurs when families of polynomials are represented by
robust arithmetic circuits or by neural network architectures as in
Sections \ref{subsec: robust arithmetic circuits} and \ref{subsec:
neural networks}). Then we may set up an algebraic computation tree
of size polynomial in $L$ which for any pair
$(v_1,v_2)\in\ol{\mathcal{N}_1}\times\ol{\mathcal{N}_2}$ decides
whether $\omega_1(v_1)=\omega_2(v_2)$ holds.

We are now going to axiomatize this situation as follows. Let $I$ be
an index set and $\cH:=\{(\theta_i,\mu_i,\omega_i)\}_{i\in I}$ be a
collection of abstraction functions $\theta_i:\cM_i \to \cO_i$
associated with geometrically robust constructible and polynomial
maps $\mu_i:\cM_i \to \cN_i$ and $\omega_i:\cN_i\to \cO_i$,
respectively. We call the collection $\cH$ {\it compatible} if for
any $i,j\in I$ there is given an algebraic computation tree of depth
polynomial in the sizes of $\theta_i$ and $\theta_j$ which for any
pair $(v_i,v_j)\in \ol{\cN_i}\times \ol{\cN_j}$ decides whether
$\omega_i(v_i)=\omega_j(v_j)$ holds.
%
%
\subsection{Quiz games}
We now exclusively consider framed unary abstract data type
carriers. We are going to model mathematically, in the geometric
context of this paper, the informal notion of \emph{information
hiding} in software engineering. For this purpose we present two
games\footnote{We remark we are not using the word ``game'' in the
sense of game theory, but rather in the informal sense of the
popular guessing game ``I spy''.}, the first one in an exact and the
second one in an approximate setting.

Let be given two abstraction functions $\theta:\cM\to\cO$,
$\theta':\cM\to\cO'$ as before, associated with geometrically robust
constructible maps $\mu:\cM\to\cN$, $\mu':\cM\to\cN'$, and
polynomial maps $\omega:\cN\to\cO$, $\omega':\cN'\to\cO'$, namely
$\theta=\omega\circ\mu$ and $\theta'=\omega'\circ\mu'$. Suppose that
$(\theta',\mu,\omega)$ and $(\theta',\mu',\omega')$ belong to a
compatible collection $\mathcal{H}$. Suppose further that there
exists a geometrically robust constructible map $\tau:\cO\to \cO'$
such that $\tau \circ \theta= \theta'$ holds. The following
commutative diagram of geometrically robust constructible maps
summarizes this situation:
\[
\xymatrix{
 & \cO \ar[r]^{\tau} & \mathcal{O}' \\
\mathcal{N} \ar[ur]^{\omega} &  \mathcal{M} \ar[l]^{\mu} \ar[r]_{\mu'} \ar[ur]^{\theta'} \ar[u]^{\theta} & \mathcal{N}' \ar[u]_{\omega'}
}
\]
We shall define two variants of a two--party protocol called the
\emph{exact} and the \emph{approximative quiz game}. The agents in
these protocols are called quizmaster and player. We suppose that
the \emph{quizmaster} has limited and the \emph{player} unlimited
computational power and that the quizmaster is honest and able to
answer the player's questions. The quizmaster is able to evaluate
$\mu'$ whereas the player can compute all occurring functions.
%
%
\subsubsection{The protocol of the exact quiz game}
\label{subsubsec: exact game}
The quizmaster chooses $u\in \cM$ and hides it from the player. The
player asks the quizmaster's questions about the polynomial
$\theta(u)$ whose answers are tuples of complex numbers which depend
only on $\theta(u)$ but not directly on $u$. The quizmaster's
answers to the player's questions are represented by a geometrically
robust map $\widetilde{\sigma}$ with domain of definition $\cO$.
Thus the arguments of $\widetilde{\sigma}$ belong to the affine
ambient space of $\mathcal{O}$.

In the aim to avoid the modeling of branchings, all the questions
and their answers become realized by a single round. In other words,
the game is not cooperative because the player's questions do not
depend on previous answers of the quizmaster.

Let $\sigma:=\widetilde{\sigma}\circ \theta$ and
$\cM^*:=\widetilde{\sigma}(\cO)=\sigma(\cM)$. We suppose now that
there is given a framed abstract data type carrier $\cO^*$ and an
abstraction function $\theta^*:\cM^*\to\cO^*$, associated with a
geometrically robust map $\mu^*:\cM^*\to\cN^*$ and a polynomial map
$\omega^*:\cN^*\to\cO^*$, from the compatible collection
$\mathcal{H}$, such that $\theta^*=\omega^*\circ\mu^*$. Moreover, we
suppose that the quizmaster is able to evaluate
$\sigma=\widetilde{\sigma}\circ \theta$.

The quizmaster computes $\sigma(u)$ and sends it to the player who
computes now $v^*:=(\mu^*\circ \sigma)(u)$. The player sends $v^*$
back to the quizmaster. The quizmaster evaluates $v':=\mu'(u)$ and
checks whether $\omega^*(v^*)=\omega'(v')$ holds. The maps
$\sigma:\cM\to\cM^*$ and $\mu^*:\cM^*\to\cN^*$ constitute the
\emph{strategy} of the player. The whole situation becomes depicted
by the following commutative diagram of constructible maps:
\[
\xymatrix{
\mathcal{N} \ar[dr]^{\omega} &  &  &  &   \\
\mathcal{M} \ar[u]^{\mu} \ar[d]_{\mu'} \ar[dr]^{\theta'} \ar[r]^{\theta}  \ar@/^2pc/[rrrr]^{\sigma} & \mathcal{O} \ar[rrr]^{\widetilde{\sigma}} \ar[d]^{\tau} &  &  & \mathcal{M}^* \ar[dl]_{\theta^*} \ar[d]^{\mu^*} \\
\mathcal{N}' \ar[r]^{\omega'} & \mathcal{O}' \ar@{}[r]|{\subset}  &
\mathcal{R} & \mathcal{O}^* \ar@{}[l]|{\supset}  & \mathcal{N}^*
\ar[l]_{\omega^*} }
\]
The player wins the instance of the game given by $u$ if the
condition $\omega^*(v^*)=\omega'(v')$ is satisfied. We say that the
player has a \emph{winning strategy} if he wins the game for any
$u\in \cM$.

We say that the computational task determined by $\theta$ and
$\theta'$ is \emph{feasible} if the size of $\theta'$ is polynomial
in the size of $\theta$. A winning strategy is called
\emph{efficient} if the size of $\theta^*$ is polynomial in the size
of $\theta$ (and therefore polynomial in the size of $\theta'$ for a
feasible computational task). Otherwise it is called
\emph{inefficient}. Observe that $\sigma$ and
$\theta^*=\omega^*\circ\mu^*$ define a winning strategy for the
exact game protocol if and only if
$\theta'=\omega^*\circ\mu^*\circ\tilde{\sigma}\circ
\theta=\omega^*\circ\mu^*\circ\sigma=\theta^*\circ\sigma $ holds. In
this case we have $\mathcal{O}'=\mathcal{O}^*$. The idea behind this
computational model is to restrict the information which quizmaster
and player may interchange. It is supposed that $\sigma(u)$ and
$v^*$ are complex vectors of short length, whereas explicit
descriptions of the framed abstract data type carriers ${\mathcal
O}$, ${\mathcal O}'$ and ${\mathcal O}^*$ may become huge.
%
%
\subsubsection{The protocol of the approximative quiz game}
\label{subsec: approximative game}
Now we introduce the protocol of the approximative quiz game. In
Appendix \ref{sec: approximative representations} below we show that
this protocol captures the notions of approximative algorithms and
approximative complexity of algebraic complexity theory (see, e.g.,
\cite[Chapter 15]{Burgisser97}).

As the game follows almost the same rules as in the exact case, we
shall therefore only stick on the differences. The quizmaster's
answers to the player's questions are now represented by a
geometrically robust constructible map $\sigma$ with domain of
definition $\cM$. We do not anymore assume that $\sigma$ is a
composition of $\theta$ with another map.

We suppose that the quizmaster is able to evaluate $\sigma$. As
before let $\cM^*:=\sigma(\cM)$ and assume that there is given a
framed abstract data type carrier $\cO^*$ and an abstraction
function $\theta^*:\cM^*\to\cO^*$, associated with a geometrically
robust constructible map $\mu^*:\cM^*\to\cN^*$ and a polynomial map
$\omega^*:\cN^*\to\cO^*$, from the compatible collection
$\mathcal{H}$, such that $\theta^*=\omega^*\circ\mu^*$.
The situation is depicted now in the following commutative diagram:
\[
\xymatrix{
\mathcal{N} \ar[dr]^{\omega} &  &  &  &   \\
\mathcal{M} \ar[u]^{\mu} \ar[d]_{\mu'} \ar[dr]^{\theta'} \ar[r]^{\theta}  \ar@/^2pc/[rrrr]^{\sigma} & \mathcal{O}  \ar[d]^{\tau} &  &  & \mathcal{M}^* \ar[dl]_{\theta^*} \ar[d]^{\mu^*} \\
\mathcal{N}' \ar[r]^{\omega'} & \mathcal{O}' \ar@{}[r]|{\subset}  &
\mathcal{R} & \mathcal{O}^* \ar@{}[l]|{\supset}  & \mathcal{N}^*
\ar[l]_{\omega^*} }
\]

We suppose that the following condition is satisfied:\medskip

\emph{For any (not necessarily convergent) sequence $(u_k)_{k\in
\N}$ in $\cM$ and $u\in \cM$ such that $(\theta(u_k))_{k\in \N}$
converges to $\theta(u)$ in the Euclidean topology, the sequence
$\big((\mu^*\circ \sigma)(u_k)\big)_{k\in \N}$ is bounded.}
\medskip

We remark that this condition is satisfied if and only if the map
$\mu^*\circ\sigma:\mathcal{M}\rightarrow\mathcal{N}^*$ is locally
bounded with respect to the Euclidean metric of $\mathcal{N}^*$ and
the topology of $\mathcal{M}$ induced from the Euclidean topology of
$\mathcal{O}$ by $\theta:\mathcal{M}\rightarrow\mathcal{O}$. As we
shall see in Proposition \ref{p:1} below, this implies that for any
$u\in\mathcal{M}$ the value $(\mu^*\circ \sigma)(u)$ strongly
depends on $\theta(u)$, although it may be not uniquely determined
by $\theta(u)$.

The quizmaster chooses now a sequence $(u_k)_{k\in \N}$ in $\cM$ and
an element $u\in \cM$ such that $(\theta(u_k))_{k\in \N}$ converges
to $\theta(u)$, and hides these data from the player. The
quizmaster's answers to the player's questions encode the sequence
$(\sigma(u_k))_{k\in \N}$, which is not necessarily convergent. From
these answers the player infers the sequence
$\big((\mu^*\circ\sigma)(u_k)\big)_{k\in \N}$, which is bounded. The
player wins the approximative game if there is an accumulation point
$v^*$ of $\big((\mu^*\circ\sigma)(u_k)\big)_{k\in \N}$ in the
Euclidean topology with $\omega^*(v^*)=\theta'(u)$. The quizmaster
verifies this by computing $v':=\mu'(u)$ and checking whether
$\omega^*(v^*)=\omega'(v')$ holds.

Recall that the computational task determined by $\theta$ and
$\theta'$ is called \emph{feasible} if the size of $\theta'$ is
polynomial in the size of $\theta$. Again we say that the maps
$\sigma:\cM\to\cM^*$ and $\mu^*:\cM^*\to\cN^*$ constitute the
\emph{strategy} of the player, that the player has a \emph{winning
strategy} if he wins for any $u\in \cM$, and we call this winning
strategy \emph{efficient} if the size of $\theta^*$ is polynomial in
the size of $\theta$ (and then polynomial in the size of $\theta'$
for a feasible computational task). Otherwise, the strategy is
called \emph{inefficient}.

\begin{lemma}
\label{observation} Suppose that $\sigma$ and
$\theta^*=\omega^*\circ\mu^*$ define a winning strategy for the
approximative quiz game protocol. Then we have
$\theta'=\omega^*\circ\mu^*\circ\sigma=\theta^*\circ\sigma$ and
therefore $\mathcal{O}'=\mathcal{O}^*$.
\end{lemma}

\begin{proof}
Let $u$ be an arbitrary point of $\cM$ and consider the
approximative game given by the sequence $(u_k)_{k\in\N}$ defined by
$u_k:=u$, and $u$. Obviously $v^*:=(\mu^*\circ\sigma)(u)$ is the
unique accumulation point of the sequence
$\big(\mu^*\circ\sigma(u_k)\big)_{k\in\N}$. Since $\sigma$ and
$\theta^*=\omega^*\circ\mu^*$ define a winning strategy, we conclude
$\theta'(u)=\omega^*(v^*)=\omega^*((\mu^*\circ\sigma)(u))=
(\omega^*\circ\mu^*\circ\sigma)(u)$. This implies the identities
$\theta'=\omega^*\circ\mu^*\circ\sigma$ and
$\mathcal{O}'=\mathcal{O}^*$.
\end{proof}

We observe that a protocol of the exact quiz game gives always rise
to a protocol of the approximative quiz game. To this end, let
notations be as in our description of the exact model and let be
given a sequence $(u_k)_{k\in \cM}$ of elements of $\cM$ and $u\in
\cM$ such that $(\theta(u_k))_{k\in \N}$ converges to $\theta(u)$.
From the continuity of $\widetilde{\sigma}$ we deduce that the
sequence $(\sigma(u_k))_{k\in \N}= \big((\widetilde{\sigma}\circ
\theta)(u_k)\big)_{k\in \N}$ converges to
$\sigma(u)=(\widetilde{\sigma}\circ \theta)(u)$. The continuity of
$\mu^*$ implies now that $\big((\mu^*\circ \sigma)(u_k)\big)_{k\in
\N}$ converges to $(\mu^*\circ \sigma)(u)$. Therefore the sequence
$\big((\mu^*\circ \sigma)(u_k)\big)_{k\in \N}$ is bounded and has a
single accumulation point, namely $v^*=(\mu^*\circ \sigma)(u)$. In
particular, the player wins the exact game given by $u\in \cM$ if
and only if he wins the approximative quiz game given by the
sequence $(u_k)_{k\in \N}$ defined by $u_k:=u$, and $u\in \cM$.

\begin{proposition}\label{p:1}
Let assumptions and notations be that of the approximative quiz game
and let $u\in \cM$. Suppose that the player has a winning strategy.
Then there exists a finite subset $\cS_u$ of the ambient space of
$\cN^*$ with the following property: for any sequence $(u_k)_{k\in
\N}$ in $\cM$ with $(\theta(u_k))_{k\in \N}$ converging to
$\theta(u)$, all the accumulation points of $\big((\mu^*\circ
\sigma)(u_k)\big)_{k\in \N}$ belong to $\cS_u$.
\end{proposition}
\begin{proof}
Let
\begin{align*}
\cD^+&:=\big\{\big(\theta(u),(\mu^*\circ \sigma)(u)\big):u\in
\cM\big\}\subset\cO\times\cN^*,\\
\cO^+&:=\big\{(\theta(u),\theta'(u)):u\in \cM\big\}\subset
\cO\times\cO',
\end{align*}
and notice that $\cD^+$ and $\cO^+$ are constructible in their
respective ambient spaces. Let $\omega^+:\mathcal{D}^+ \rightarrow
\mathcal{O}^+$ be the polynomial map
$\omega^+:=(\mathrm{id},\omega^*)$, namely
\[
\omega^+\big(\theta(u), (\mu^*\circ
\sigma)(u)\big):=\big(\theta(u),(\omega^*\circ \mu^*\circ
\sigma)(u)\big)=\big(\theta(u), (\theta^*\circ \sigma)(u)\big).
\]
By Lemma \ref{observation}, the assumption that the player has a
winning strategy implies that $(\theta^*\circ\sigma)(u)=\theta'(u)$
holds for any $u\in\mathcal{M}$. Hence $\omega^+$ is surjective.

We show that the surjective polynomial map $\omega^+$ satisfies the
condition of Lemma \ref{lemma module} in Appendix \ref{section:
futher facts on geom rob maps}. To this end, consider a sequence
$(u_k)_{k\in \N}$ in $\cM$ and an element $u\in \cM$ such that
$\big((\theta(u_k),\theta'(u_k))\big)_{k\in \N}$ converges to
$\big(\theta(u),\theta'(u)\big)$. Then $(\theta(u_k))_{k\in \N}$
converges to $\theta(u)$, which implies that $\big((\mu^*\circ
\sigma)(u_k)\big)_{k\in \N}$, and therefore $\big(\theta(u_k),
(\mu^*\circ\sigma)(u_k)\big)_{k\in\N}$, are bounded.

Let $u\in\mathcal{M}$ and let $(u_k)_{k\in \N}$ be an arbitrary
sequence in $\cM$ such that $(\theta(u_k))_{k\in \N}$ converges to
$\theta(u)$. Observe that $\omega^+$ induces a $\C$-algebra
extension $\C[\ol{\cO^+}]\hookrightarrow \C[\ol{\cD^+}]$. Let $\m$
be the maximal ideal of definition of the point $\big(\theta(u),
\theta'(u)\big)$ of the affine variety $\ol{\cO^+}$. From Lemma
\ref{lemma module} we deduce that the
$\C[\ol{\mathcal{O}^+}]_{\mathfrak{m}}$--module
$\C[\ol{\cD^+}]_{\mathfrak{m}}$ is finite. Hence the $\C$-algebra
extension $\C[\ol{\cO^+}]_{\m}\hookrightarrow \C[\ol{\cD^+}]_\m$ is
integral.

Let $\lambda$ be the coordinate function of $\ol{\cD^+}$
corresponding to an arbitrary entry of $\mu^*\circ \sigma$. Then
there exists $s\in\N$ and elements $a_0,a_1\klk a_s$ of
$\C[\ol{\cO^+}]$ with $a_0\notin \m$ such that the algebraic
dependence relation $a_0\lambda^s+a_1\lambda^{s-1}+\cdots +a_s=0$ is
satisfied in $\C[\ol{\cD^+}]$. This implies
\begin{equation}
\label{e:1} \sum_{j=0}^sa_j\big(\theta(u_k),
\theta'(u_k)\big)\lambda^{s-j} \big(\theta(u_k),(\mu^*\circ
\sigma)(u_k)\big)= 0
\end{equation}
for any $k\in \N$. On the other hand, as $a_0\notin \m$, we deduce
that $a_0\big(\theta(u), \theta'(u)\big)\neq 0$ holds. Let $v^*$ be
an accumulation point of the sequence $\big((\mu^*\circ
\sigma)(u_k)\big)_{k\in \N}$. From (\ref{e:1}) we conclude that
$$
\sum_{j=0}^sa_j\big(\theta(u),
\theta'(u)\big)\lambda^{s-j}(\theta(u), v^*)= 0.
$$
Since $a_0(\theta(u), \theta'(u))\neq 0$, only finitely many values
of $\lambda(\theta(u),v^*)$ satisfy this equation, independently of
the sequence $(u_k)_{k\in \N}$. Since $\lambda$ was the coordinate
function of $\ol{\cD^+}$ corresponding to an arbitrary entry of
$\mu^*\circ \sigma$, the conclusion of Proposition \ref{p:1}
follows.
\end{proof}
\begin{corollary}\label{c:1}
Let assumptions and notations be that of the approximative quiz
game. Suppose that the player has a winning strategy. Then for any
$u\in\mathcal{M}$, the set
\[\{\mu^*\circ \sigma(v):v\in \cM, \theta(v)=\theta(u)\}\]
is finite.
\end{corollary}
\begin{proof}
Let $v\in \cM$ with $\theta(v)=\theta(u)$ and let $(v_k)_{k\in \N}$
be the sequence defined by $v_k:=v$ for any $k\in \N$. Then
$(\theta(v_k))_{k\in \N}$ converges to $\theta(v)$ and $(\mu^*\circ
\sigma)(v)$ is an accumulation point of $\big((\mu^*\circ
\sigma)(v_k)\big)_{k\in \N}$. The assertion follows now from
Proposition \ref{p:1}.
\end{proof}
%
%
\subsubsection{Abstract datatype carriers, classes and routines}
The content of this subsection is aimed to clarify our conceptual
system from the point of view of software engineering and will not
be used in the sequel. Our terminology is borrowed from
\cite{Meyer97}.

A subset of $\mathcal{R}$ which can be countably covered by framed
unary data type carriers is called a \emph{unary abstract data type
carrier}. A unary abstract data type carrier together with any such
covering is called a \emph{unary class}. Mutatis mutandis we may
define the notions of \emph{$r$--ary abstract data type carrier} and
\emph{$r$--ary class} for any $r\in\N$ replacing in the above
description $\mathcal{R}$ by its cartesian product $\mathcal{R}^r$.

Let $\mathcal{U}_1$ and $\mathcal{U}_2$ be unary abstract data type
carriers. An \emph{abstract function} between $\mathcal{U}_1$ and
$\mathcal{U}_2$ is a map $\tau:\mathcal{U}_1 \to \mathcal{U}_2$ with
the following properties. For each framed data structure
$\mathcal{M}_1$ and framed unary abstract data type carrier
$\mathcal{O}_1$ together with an abstraction function
$\theta_1:\mathcal{M}_1\to\mathcal{O}_1$, such that $\mathcal{O}_1$
is contained in $\mathcal{U}_1$, there exists a framed data
structure $\mathcal{M}_2$ and a framed unary abstract data type
carrier $\mathcal{O}_2$ together with an abstraction function
$\theta_2 :\mathcal{M}_2\to\mathcal{O}_2$, such that
$\tau(\mathcal{O}_1)\subset \mathcal{O}_2 \subset \mathcal{U}_2$
holds and there exists a geometrically robust constructible map,
called a \emph{routine}, $\sigma:\mathcal{M}_1\to\mathcal{M}_2$ with
$\tau\circ\theta_1=\theta_2\circ\sigma$. Observe that quiz games
with winning strategies instantiate these properties.

For any $r_1,r_2\in\N$, this notion of abstract function between
given $r_1$ and $r_2$--ary abstract data type carriers may be
generalized if we replace the ring $\mathcal{R}$ by its cartesian
products $\mathcal{R}^{r_1}$ and $\mathcal{R}^{r_2}$, respectively.
We speak then about an \emph{abstract function with $r_1$ inputs and
$r_2$ outputs}. This concept of abstract function is not exactly
that of \cite{Meyer97}, but comes close to it.
%
%
\subsubsection{Quiz games and information hiding in software engineering}
Let assumptions and notations be as before. Suppose that there is a
programmer whose task is, in an object oriented manner, to implement
the geometrically robust constructible map
$\tau:\mathcal{O}\to\mathcal{O}'$ between the framed abstract data
type carriers $\mathcal{O}$ and $\mathcal{O}'$ which are represented
by the abstraction functions $\theta$ and $\theta'$.

His program uses observers and constructors. Information hiding
means that the objects from $\mathcal{M}$ which represent elements
of $\mathcal{O}$ and $\mathcal{O}'$ are not accessible to him. He is
only allowed to ask questions about these elements which involve
complex values that he then processes. This situation becomes
modeled by the exact quiz game with the programmer in the r\^{o}le
of the player and the observers in the r\^{o}le of the quizmaster.

More subtle is the situation in the case of the approximative quiz
game. Here we allow the programmer to access to a limited extent
information about the objects which represent the elements of
$\mathcal{O}$ and $\mathcal{O}'$. This information is of
approximative nature and becomes provided by the observers. The
programmer/player must then be able, by means of constructors, to
compute a representation of the output. From Proposition \ref{p:1}
we deduce that the situation  modeled by the approximative quiz game
is not substantially different from that modeled by the exact one.
For each input element of $\mathcal{O}$ there exist only finitely
many possible representations of the corresponding output element of
$\mathcal{O}'$ which can be computed by the programmer/player.
%
%
\section{Selected complexity lower bounds}
We are now going to use the computational model developed in
Sections \ref{subsubsec: exact game} and \ref{subsec: approximative
game} to derive complexity lower bounds for selected computational
problems.

From the seven series of examples we are going to consider in this
section only three are really new. The other ones can be found in
another context in \cite{CaGiHeMaPa03}, \cite{GiHeMaSo11},
\cite{HeKuRo} and \cite{HKRPSantalo}. What is new is our unified
approach to them and the resulting generality of our complexity
results.
%
%
\subsection{Polynomial Interpolation}
In this section we are going to consider four types of approximative
quiz game protocols which generalize the intuitive meaning of
interpolation of families of multivariate and univariate
polynomials.
%
%
\subsubsection{Multivariate polynomial interpolation}
\label{subsubsec: multivariate polynomial interpolation}
Let be given a framed abstract data type carrier $\mathcal{O}$,
framed data structures $\mathcal{M}$ and $\mathcal{N}$, an
abstraction function $\theta:\mathcal{M}\rightarrow\mathcal{O}$
associated with a geometrically robust constructible map
$\mu:\mathcal{M}\rightarrow\mathcal{N}$ and a polynomial map
$\omega:\mathcal{N}\rightarrow\mathcal{O}$, and suppose that
$(\theta,\mu,\omega)$ belongs to a given compatible collection
$\mathcal{H}$.

Let $\mathcal{O}':=\mathcal{O}$, $\tau:=\text{id}_{\mathcal{O}}$,
$\theta':=\theta$, $\mu':=\mu$, $\omega':=\omega$ and let be given
an approximative quiz game protocol with winning strategy for this
situation. Then the player returns for a given parameter instance
$u\in\mathcal{M}$, a point $v^*$ of a suitable affine space such
that $v^*$ encodes $\theta(u)$. In other words, the player replaces
the ``hidden'' encoding $u$ of $\theta(u)$ by a new encoding $v^*$
which he computes from the quizmaster's answers to his questions.

In the exact version of the quiz game the player is limited to ask
questions about the polynomial $\theta(u)$ itself and computes from
the quizmaster's answers his new encoding of $\theta(u)$. If the
player's questions refer only to values of the polynomial
$\theta(u)$ at given inputs, the player solves an interpolation
problem for $\theta(u)$. Below we are going to make more precise
this aspect.
%
%
\paragraph{Interpolation of a family easy to evaluate.}
Let us now analyze the following concrete example of this general
setting. Let $l,n\in\N$ be discrete parameters with
$2^{\frac{l}{2}}\geq n$ and let $\mathcal{M}:=\A^{n+1}$. For
$t\in\A^1$ and $u=(u_1,\dots,u_n)\in\A^n$, let $\theta(t,u)$ be the
polynomial
$$\theta(t,u):=t\sum_{k=0}^{2^l-1} (u_1X_1 + \dots +
u_nX_n)^k,$$
and let $\mathcal{O}:=\text{im\ }\theta$. Observe that the family of
polynomials $\theta$ is evaluable by a robust arithmetic circuit of
size $2n+3l-1$ and that, for any $t\in\A^1$ and any $u\in\A^n$, the
polynomial $\theta(t,u)$ can be computed using $2l-2$ essential
multiplications. Thus there exists an injective affine linear map
$\mu:\mathcal{M}\rightarrow\A^{2n+3l-1}$ with constructible image
$\mathcal{N}:=\mu(\mathcal{M})$ and a polynomial map
$\omega:\mathcal{N}\rightarrow\mathcal{O}$ such that
$\theta=\omega\circ\mu$ holds. Hence $\mathcal{O}$ is a framed
abstract data type carrier and
$\theta:\mathcal{M}\rightarrow\mathcal{O}$ is an abstraction
function of size $2n+3l-1$ associated with $\mu$ and $\omega$. From
Proposition \ref{propo M}, and the comment that follows it, we
deduce that $(\theta,\mu,\omega)$ belongs to a suitable collection
$\mathcal{H}$ of abstraction functions.

Suppose that for the previously considered computation task given by
$\theta$ and id$_{\mathcal{O}}$ there is given a protocol with
winning strategy for the approximative quiz game. Thus we have a
framed abstract data type carrier $\mathcal{O}^*$, framed data
structures $\mathcal{M}^*$ and $\mathcal{N}^*$, an abstraction
function $\theta^*:\mathcal{M}^*\rightarrow\mathcal{O}^*$ of the
compatible collection $\mathcal{H}$, associated with a geometrically
robust constructible map $\mu^*:\mathcal{M}^*\to\mathcal{N}^*$ and a
polynomial map $\omega^*:\mathcal{N}^*\rightarrow\mathcal{O}^*$, and
a surjective geometrically robust constructible map
$\sigma:\mathcal{M}\rightarrow\mathcal{M}^*$ such that by Lemma
\ref{observation} the identities
$\theta=\omega^*\circ\mu^*\circ\sigma=\theta^*\circ\sigma$ and
$\mathcal{O}=\mathcal{O}^*$ hold. We summarize the whole situation
by the following commutative diagram of geometrically robust
constructible maps:
\[
\xymatrix{
 & \mathcal{O} \ar@{}[r]|{=} &  \mathcal{O}^* &    \\
\mathcal{N} \ar[ur]^{\omega} & \mathcal{M} \ar[l]_{\mu}
\ar[u]_{\theta} \ar[r]^{\sigma} &  \mathcal{M}^* . \ar[u]^{\theta^*}
\ar[r]^{\mu^*} & \mathcal{N}^* \ar[ul]_{\omega^*}   }
\]

With these notations and assumptions, we have the following
statement.
\begin{lemma}
\label{lemma order}
The size of $\theta^*$ is of order $2^{\Omega(ln)}$.
\end{lemma}
\begin{proof}
Following \cite[Section 5.2]{GiHeMaSo11}, one verifies easily for
any $t\in\A^1$ and $u=(u_1,\dots,u_n)\in\A^n$ the following
identities:
\begin{align*}
\theta(t,u)&= t\sum_{k=0}^{2^l-1}
\sum_{\overset{(\alpha_1,\dots,\alpha_n) \in \mathbb{Z}_{\ge
0}^n}{\scriptscriptstyle \alpha_1+\dots+\alpha_n = k}}
\frac{k!}{\alpha_1!\dots \alpha_n!}\, u_1^{\alpha_1}\dots
u_n^{\alpha_n}   X_1^{\alpha_1}\dots X_n^{\alpha_n}
\\
&=t\sum_{k=0}^{2^l-1}
\sum_{\overset{(\alpha_1,\dots,\alpha_n)\in\Z_{\geq
0}^{n}}{\scriptscriptstyle \alpha_1+\dots+\alpha_n < 2^l}}
\frac{(\alpha_1+\dots+\alpha_n)!}{\alpha_1!\dots \alpha_n!}
\,u_1^{\alpha_1}\dots u_n^{\alpha_n} X_1^{\alpha_1}\dots
X_n^{\alpha_n}.
\end{align*}

For $(\rho,t)\in\A^2$, let
$\beta_{\rho}(t):=(t,\rho,\rho^{2^l},\dots,\rho^{2^{(n-1)l}})$.
Then, for fixed $\rho\in\A^1$, the map which assigns to each
$t\in\A^1$ the value $(\theta\circ\beta_{\rho})(t)$ is represented
by the polynomial
\[
(\theta\circ\beta_\rho)(t)=t
\sum_{\overset{(\alpha_1,\dots,\alpha_n)\in\Z_{\geq
0}^{n}}{\scriptscriptstyle \alpha_1+\dots+\alpha_n < 2^l}}
\frac{(\alpha_1+\dots+\alpha_n)!}{\alpha_1!\dots \alpha_n!}
\,\rho^{\alpha_1+\alpha_2 2^l +\dots +\alpha_n 2^{(n-1)l} }
X_1^{\alpha_1}\dots X_n^{\alpha_n},
\]
and is therefore linear in $t$.

Observe that the set ${M}_l:=\{(\alpha_1,\dots,\alpha_n)\in\Z_{\geq
0}^{n} :\alpha_1+\dots+\alpha_n < 2^l \}$ has
$K:=\binom{2^l-1+n}{n}$ elements. Moreover, each element
$\underline{\alpha}=(\alpha_1,\dots,\alpha_n)$ of this set can be
identified with its $2^l$--ary representation $\alpha_1+\alpha_2 2^l
+\dots +\alpha_n 2^{(n-1)l} $. From \cite[Lemma 24]{GiHeMaSo11} we
deduce that there exists a non--empty Zariski open subset $\cU$ of
$\A^K$ such that for any point $(\rho_1,\dots,\rho_K)\in\cU$, the
$(K\times K)$--matrix
\[
{\tt
M}_{\rho_1\klk\rho_K}:=\left(\frac{(\alpha_1+\dots+\alpha_n)!}{\alpha_1!\dots
\alpha_n!} \ \rho_s^{\alpha_1+\alpha_2 2^l +\dots+ \alpha_n
2^{(n-1)l} } \right)_{\underline{\alpha}=(\alpha_1\klk\alpha_n)\in
M_l,\,1\leq s\leq K}
\]
is nonsingular. Hence, for $(\rho_1,\dots,\rho_K)\in\cU$, the
vectors $\frac{d}{dt}
(\theta\circ\beta_{\rho_1})(0),\dots,\frac{d}{dt}
(\theta\circ\beta_{\rho_K})(0)$ are $\C$--linearly independent.

The map which assigns to each $(\rho,t)\in\A^2$ the value
$(\mu^*\circ \sigma\circ \beta_{\rho})(t)$ is geometrically robust
and constructible and hence polynomial by Theorem \ref{theorem
geometrically robust functions}$(iii)$. Since any $u\in\A^n$
satisfies the condition $\theta(0,u)=0$ we conclude from Corollary
\ref{c:1} that the set $\{(\mu^*\circ\sigma)(0,u):u\in\A^n \}$ is
finite. Therefore the set $\{(\mu^*\circ\sigma\circ\beta_{\rho})(0):
\rho\in\A^1 \}$ is also finite. Hence there exists a cofinite subset
$\Gamma$ of $\C$ such that for any point $\rho\in\Gamma$, the value
$(\mu^*\circ\sigma\circ\beta_{\rho})(0)$ is the same, say $v$.

On the other hand, $\mathcal{U} \cap \Gamma ^{K}$ is a nonempty
Zariski open subset of $\A^n$, because $\Gamma$ is cofinite.
Therefore there exist elements $\rho_1,\dots,\rho_K\in\Gamma$ such
that the vectors $\frac{d}{dt}
(\theta\circ\beta_{\rho_1})(0),\dots,\frac{d}{dt}(\theta\circ\beta_{\rho_K})(0)$
are $\C$--linearly independent. For $1 \leq j \leq K$, the map
$t\mapsto(\mu^*\circ\sigma\circ\beta_{\rho_j})(t)$ is polynomial and
hence holomorphic. The assumption that the player has a winning
strategy for the given approximative quiz game implies
$$(\theta\circ\beta_{\rho_j})(t)=
(\omega^*\circ\mu^*\circ\sigma\circ\beta_{\rho_j})(t)$$
for any $t\in\A^1$. We may now apply the chain rule to this
situation to conclude
\begin{align*}
\frac{d}{dt}(\theta\circ\beta_{\rho_j})(0)&=(d\omega^*)
\big((\mu^*\circ\sigma\circ\beta_{\rho_j})(0)\big)\cdot \frac{d}{dt}
(\mu^*\circ\sigma\circ\beta_{\rho_j})(0)\\
&= (d \omega^*)(v)\cdot
\frac{d}{dt}(\mu^*\circ\sigma\circ\beta_{\rho_j})(0),
\end{align*}
where $d\omega^*$ denotes the total differential of the polynomial
map $\omega^*$. Since
$\frac{d}{dt}(\theta\circ\beta_{\rho_1})(0)\klk
\frac{d}{dt}(\theta\circ\beta_{\rho_K})(0)$ are $\C$--linearly
independent, we infer that the vectors
$\frac{d}{dt}(\mu^*\circ\sigma\circ\beta_{\rho_1})(0),
\dots,\frac{d}{dt}(\mu^*\circ\sigma\circ\beta_{\rho_K})(0)$ must
generate a $\C$--linear space of dimension $K$. This and the
assumption $ 2^{\frac{l}{2}}\geq n$ imply now that the size of
$\mathcal{N}^*$, and hence the size of $\theta^*$, is at least
$K=\binom{2^l -1+n}{n}=2^{\Omega(ln)}$.
\end{proof}

In conclusion, the winning strategy of the approximative quiz game
under consideration is necessarily inefficient. We formulate now
this conclusion in a more general setting.
\begin{theorem}
\label{theorem lower bound} Let $L,n\in\N$ with $2^{\frac{L}{4}}\geq
n$ and let $\mathcal{O}_{L,n}$ be the abstract data type of all
polynomials of $\C[X_1,\dots,X_n]$ which can be evaluated using at
most $L$ essential multiplications (see Section \ref{subsec: robust
arithmetic circuits}). We think the abstraction function of
$\mathcal{O}_{L,n}$ represented by the corresponding generic
computation and consider the task of replacing the given hidden
encoding of the elements of $\mathcal{O}_{L,n}$ by a known one using
an approximative quiz game with winning strategy. Then any such quiz
game is inefficient requiring an abstraction function of size
$2^{\Omega(Ln)}$.
\end{theorem}
\begin{proof}
Let $l:= \lfloor \frac{L}{2} + 1 \rfloor$ and observe that for any
$t\in\A^1$ and $u\in\A^n$ the polynomial $\theta(t,u)$ belongs to
$\mathcal{O}_{L,n}$. Then the above statement follows immediately
from Theorem \ref{theorem geometrically robust functions}$(i)$ and
Lemma \ref{lemma order}.
\end{proof}

The restriction of Theorem \ref{theorem lower bound} to exact quiz
games with winning strategy may be paraphrased in terms of learning
theory as follows.
\begin{corollary}
For $L,n\in\N$ with $2^{\frac{L}{4}}\geq n$, the concept class
$\mathcal{O}_{L,n}$ has a representation of size $(L+n+1)^2$ by
means of the corresponding generic computation, but its concepts
require an amount of $2^{\Omega(Ln)}$ arithmetic operations to be
learned continuously.
\end{corollary}
%
%
\paragraph{Interpolation from an identification sequence.}
We are now going to explain how the computational task of
interpolating for $L,n\in\N$ the family of polynomials
$\mathcal{O}_{L,n}$ \emph{continuously} may be interpreted as a
particular instance of an exact quiz game protocol. From Proposition
\ref{propo M} we deduce that there exist for
$\mathcal{O}_{L,n}\times\mathcal{O}_{L,n}$ (many) integer
identification sequences of length $m:=4(L+n+1)^2 +2$ and bit size
at most $3L+1$. Let us fix for the moment such an identification
sequence $\gamma=(\gamma_1,\dots,\gamma_m)\in(\Z^n)^m$ and let
$\widetilde{\sigma}:\mathcal{O}_{L,n}\rightarrow\A^m$ be the map
$\widetilde{\sigma}(f):=(f(\gamma_1),\dots,f(\gamma_m))$. Then
$\widetilde{\sigma}$ is a polynomial map with constructible image
$\mathcal{M}^*:=\widetilde{\sigma}(\mathcal{O}_{L,n})$ and
$\widetilde{\sigma}:\mathcal{O}_{L,n}\rightarrow\mathcal{M}^*$ is
bijective. Let $\mathcal{N}^*$ be the constructible subset of $\A^{
\binom {2^L +n} {n}}$ formed by the coefficient vectors of the
elements of $\mathcal{O}_{L,n}$ and let
$\omega^*:\mathcal{N}^*\rightarrow\mathcal{O}_{L,n}$ be the map
which assigns to each element of $\mathcal{N}^*$ the corresponding
polynomial of $\C[X_1,\dots,X_n]$. Then $\omega^*$ is a polynomial
map and $\mathcal{M}^*$ and $\mathcal{N}^*$ form framed data
structures. According to \cite[Corollary 3]{CaGiHeMaPa03}, there is
a unique rational, everywhere defined, constructible map
$\mu^*:\mathcal{M}^*\rightarrow\mathcal{N}^*$, which is continuous
with respect to the Euclidean topologies of $\mathcal{M}^*$ and
$\mathcal{N}^*$, whose composition $\omega^*\circ\mu^*$ with
$\omega^*$ is the inverse map of
$\widetilde{\sigma}:\mathcal{O}_{L,n}\rightarrow\mathcal{M}^*$.

From Theorem \ref{theo continuos} we deduce that $\mu^*$ is
geometrically robust. Let us consider the framed abstract data type
carriers $\mathcal{O}:=\mathcal{O}_{L,n}$ and
$\mathcal{O}^*:=\mathcal{O}_{L,n}$, the framed data structures
$\mathcal{M}:=\A^{(L+n+1)^2}$,
$\mathcal{M}^*:=\widetilde{\sigma}(\cO_{L,n})$ and $\mathcal{N}^*$,
the polynomial map $\theta:\mathcal{M}\rightarrow\mathcal{O}$ given
by the generic computation corresponding to $\mathcal{O}_{L,n}$, the
polynomial map $\omega^*:\mathcal{N}^*\rightarrow\mathcal{O}^*$ and
the geometrically robust constructible maps
$\mu^*:\mathcal{M}^*\rightarrow\mathcal{N}^*$,
$\theta^*:=\omega^*\circ\mu^*$,
$\widetilde{\sigma}:\mathcal{O}\rightarrow\mathcal{M}^*$ and
$\sigma:=\widetilde{\sigma}\circ\theta$. They form the following
commutative diagram of geometrically robust constructible maps:
\[
\xymatrix{
\mathcal{O} \ar@{}[r]|{=} \ar[dr]^{\widetilde{\sigma}} & \mathcal{O}^* &  \\
\mathcal{M} \ar[u]^{\theta} \ar[r]^{\sigma} & \mathcal{M}^*
\ar[u]_{\theta^*} \ar[r]^{\mu^*} & \mathcal{N}^* \ar[ul]_{\omega^*}
}
\]
This diagram represents an exact quiz game protocol with winning
strategy for the situation considered at the beginning of this
section.

On the other hand, the geometrically robust constructible map
$\mu^*:\mathcal{M}^*\rightarrow\mathcal{N}^*$ assigns to each
element $\widetilde{\sigma}(f)=(f(\gamma_1),\dots,f(\gamma_m))$ of
$\mathcal{M}^*$ the coefficient vector
$\mu^*(\widetilde{\sigma}(f))$ of the polynomial $f$, and therefore
interpolates $f$ at the points $\gamma_1,\dots,\gamma_m\in\Z^n$.
Since $\mu^*$ is continuous with respect to the Euclidean topologies
of $\mathcal{M}^*$ and $\mathcal{N}^*$, it solves the corresponding
interpolation task continuously. This solution is also effective
because $\mu^*$ is constructible.

Given any identification sequence $\gamma=(\gamma_1,\dots,\gamma_m)$
for $\mathcal{O}_{L,n}\times \mathcal{O}_{L,n}$, any continuous
solution of the task to interpolate the elements of
$\mathcal{O}_{L,n}$ in the points $\gamma_1,\dots,\gamma_m$ may be
depicted as in the commutative diagram above with $\mu^*$ and
$\omega^*$ representing a robust arithmetic circuit. Theorem
\ref{theorem lower bound} implies now that such a solution requires
necessarily the use of $2^{\Omega(Ln)}$ arithmetic operations. This
is the content of \cite[Theorem 23]{GiHeMaSo11}.
%
%
\subsubsection{Univariate polynomial interpolation}
We are now going to analyze the three series of examples which
generalize univariate polynomial interpolation. To this end, let us
fix a discrete parameter $D\in\N$. Let $l:=6 \lceil\log D-1\rceil
+1$, $\mathcal{M}:=\A^1$ and $X$ be an indeterminate. For
$t\in\mathcal{M}$, let
$$\theta_D(t):=(t^{D+1}-1)\sum_{0\leq k \leq D} t^k X^k.$$

We put $\mathcal{O}_D:=\text{im\ }\theta_D$. Observe that the family
of polynomials $\theta_D$ is evaluable by a robust arithmetic
circuit of size $l$. Thus there exists an injective affine map
$\mu_D:\mathcal{M}\rightarrow\A^l$ with constructible image
$\mathcal{N}_D:=\mu_D(\mathcal{M})$ and a polynomial map
$\omega_D:\mathcal{N}_D\rightarrow\mathcal{O}_D$ such that
$\theta_D=\omega_D\circ\mu_D$ holds. Hence $\mathcal{O}_D$ is a
framed abstract data type carrier and
$\theta_D:\mathcal{M}\rightarrow\mathcal{O}_D$ is an abstraction
function of size $l$ associated with $\mu_D$ and $\omega_D$. From
Proposition \ref{propo M}, and the comment that follows it, we
deduce that $(\theta_D,\mu_D,\omega_D)$ belongs to a suitable
collection $\mathcal{H}$ of abstraction functions.

For any $t\in\mathcal{M}$, let $\theta_D'(t):=\theta_D(t)$,
$\theta_D''(t)$ the derivative of $\theta_D(t)$ with respect of the
variable $X$, and $\theta_D'''(t)$ the indefinite integral with
respect to $X$, namely
$$\theta_D''(t):=(t^{D+1}-1)\sum_{1\leq k \leq D}
kt^{k}X^{k-1},\quad \theta_D'''(t):=(t^{D+1}-1)\sum_{0\leq k \leq D}
\frac{t^k}{k+1} X^{k+1}.$$
Let $\mathcal{O}_D':=\text{im\ }\theta_D'$,
$\mathcal{O}_D'':=\text{im\ }\theta_D''$ and
$\mathcal{O}_D''':=\text{im\ }\theta_D'''$. We consider the
geometrically robust constructible maps
$$
\begin{array}{rcrcrclrcl}
\tau_D':&\mathcal{O}_D\rightarrow\mathcal{O}_D',&
\tau_D'':&\mathcal{O}_D\rightarrow\mathcal{O}_D'',&
\tau_D''':&\mathcal{O}_D\rightarrow\mathcal{O}_D''',\\
&\theta_D(t)\mapsto\theta_D'(t),&&\theta_D(t)\mapsto\theta_D''(t),
&&\theta_D(t)\mapsto\theta_D'''(t).
\end{array}
$$
Suppose that there are given framed data structures
$\mathcal{N}_D''$ and $\mathcal{N}_D'''$ of size polynomial in $l$,
geometrically robust constructible maps
$\mu_D'':\mathcal{M}\rightarrow\mathcal{N}_D''$ and
$\mu_D''':\mathcal{M}\rightarrow\mathcal{N}_D'''$, and polynomials
maps $\omega_D'':\mathcal{N}_D''\rightarrow\mathcal{O}_D''$ and
$\omega_D''':\mathcal{N}_D'''\rightarrow\mathcal{O}_D'''$, such that
$\theta_D'':\mathcal{M}\rightarrow\mathcal{O}_D''$ and
$\theta_D''':\mathcal{M}\rightarrow\mathcal{O}_D'''$ are abstraction
functions of the compatible collection $\mathcal{H}$ associated with
$\mu_D''$, $\omega_D''$ and $\mu_D'''$, $\omega_D'''$, respectively.
Finally, let $\mathcal{N}_D':=\mathcal{N}_D$, $\mu_D':=\mu_D$ and
$\omega_D':=\omega_D$. Each of the three items $(\theta_D,\tau_D')$
and $(\theta_D,\tau_D'')$ and $(\theta_D,\tau_D''')$ define a
computation task for an exact quiz game. The following diagram
illustrates this scenario. Notice that $\theta_D$, $\theta_D'$,
$\theta_D''$ and $\theta_D'''$ have the same domain of definition
$\mathcal{M}$.
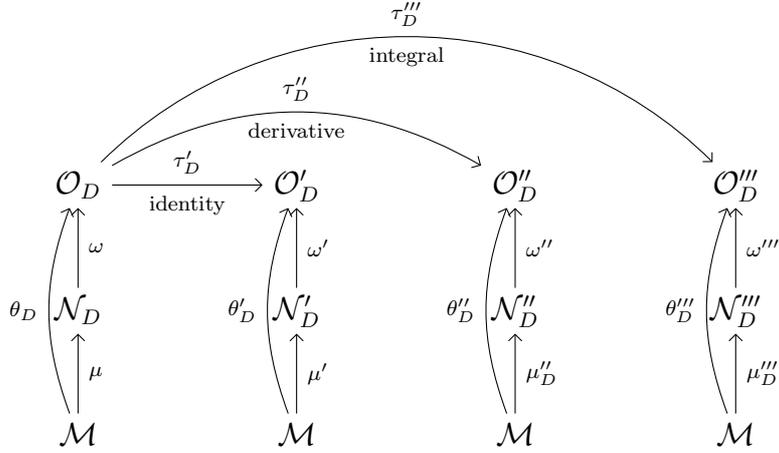
\begin{center}
\begin{tikzpicture}
\centering
\matrix(m)[matrix of math nodes,
row sep=2.6em, column sep=4.8em,
text height=1.5ex, text depth=0.25ex]
{ \mathcal{O}_D & \mathcal{O}'_D & \mathcal{O}''_D & \mathcal{O}'''_D \\
  \mathcal{N}_D & \mathcal{N}'_D & \mathcal{N}''_D & \mathcal{N}'''_D \\
  \mathcal{M}   & \mathcal{M}    & \mathcal{M}     & \mathcal{M}      \\};
\path[->,font=\scriptsize,>=angle 90]
(m-1-1) edge node[auto] {$\tau'_D$}
             node[below] {identity} (m-1-2)
        edge[out=30,in=150] node[auto] {$\tau''_D$}
                            node[below] {derivative} (m-1-3)
        edge[out=45,in=135] node[auto] {$\tau'''_D$}
                            node[below] {integral} (m-1-4)
(m-2-1) edge node[right] {$\omega$} (m-1-1)
(m-2-2) edge node[right] {$\omega'$} (m-1-2)
(m-2-3) edge node[right] {$\omega''$} (m-1-3)
(m-2-4) edge node[right] {$\omega'''$} (m-1-4)

(m-3-1) edge node[right] {$\mu$} (m-2-1)
(m-3-2) edge node[right] {$\mu'$} (m-2-2)
(m-3-3) edge node[right] {$\mu''_D$} (m-2-3)
(m-3-4) edge node[right] {$\mu'''_D$} (m-2-4)

(m-3-1) edge[out=110,in=250] node[auto] {$\theta_D$} (m-1-1)
(m-3-2) edge[out=110,in=250] node[auto] {$\theta'_D$} (m-1-2)
(m-3-3) edge[out=110,in=250] node[auto] {$\theta''_D$} (m-1-3)
(m-3-4) edge[out=110,in=250] node[auto] {$\theta'''_D$} (m-1-4);
\end{tikzpicture}
\end{center}

Suppose that for any of these tasks, e.g., for that given by
$(\theta_D,\tau_D')$, there is given an exact quiz game protocol
with winning strategy. Thus we have a framed abstract data type
carrier $\mathcal{O}^*$, framed data structures $\mathcal{M}^*$ and
$\mathcal{N}^*$, an abstraction function
$\theta^*:\mathcal{M}^*\rightarrow\mathcal{O}^*$ of the compatible
collection $\mathcal{H}$, associated with a geometrically robust
constructible map $\mu^*:\mathcal{M}^*\rightarrow\mathcal{N}^*$ and
a polynomial map $\omega^*:\mathcal{N}^*\rightarrow\mathcal{O}^*$,
and geometrically robust constructible maps
$\widetilde{\sigma}:\mathcal{O}_D\rightarrow\mathcal{M}^*$ and
$\sigma:=\widetilde{\sigma}\circ\theta_D$ such that the identities
$\theta_D'=\omega^*\circ\mu^*\circ\widetilde{\sigma}\circ\theta_D=
\omega^*\circ\mu^*\circ\sigma=\theta^*\circ\sigma$ and
$\mathcal{O}_D'=\mathcal{O}^*$ hold. The following diagram
illustrates this exact quiz game protocol.

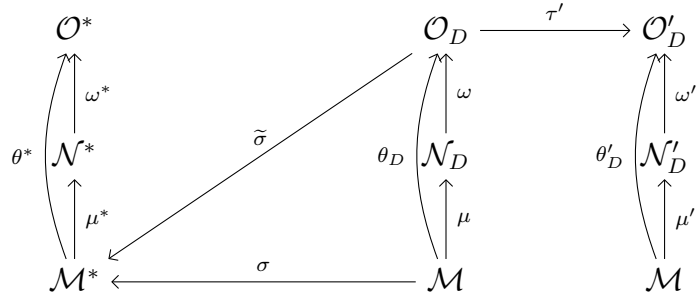
\begin{center}
\begin{tikzpicture}
\matrix(m)[matrix of math nodes,
row sep=2.6em, column sep=4.8em,
text height=1.5ex, text depth=0.25ex]
{ \mathcal{O}^* &  & \mathcal{O}_D & \mathcal{O}'_D \\
  \mathcal{N}^* &  & \mathcal{N}_D & \mathcal{N}'_D \\
  \mathcal{M}^* &  & \mathcal{M}   & \mathcal{M} \\};
\path[->,font=\scriptsize,>=angle 90]
(m-3-1) edge[out=110,in=250] node[auto] {$\theta^*$} (m-1-1)
(m-3-3) edge[out=110,in=250] node[auto] {$\theta_D$} (m-1-3)
(m-3-4) edge[out=110,in=250] node[auto] {$\theta'_D$} (m-1-4)

(m-2-1) edge[auto] node[right] {$\omega^*$} (m-1-1)
(m-3-1) edge[auto] node[right] {$\mu^*$} (m-2-1)

(m-2-3) edge[auto] node[right] {$\omega$} (m-1-3)
(m-3-3) edge[auto] node[right] {$\mu$} (m-2-3)

(m-2-4) edge[auto] node[right] {$\omega'$} (m-1-4)
(m-3-4) edge[auto] node[right] {$\mu'$} (m-2-4)

(m-1-3) edge[auto] node[auto] {$\tau'$} (m-1-4)
(m-1-3) edge[auto] node[above] {$\widetilde{\sigma}$} (m-3-1)
(m-3-3) edge[auto] node[above] {$\sigma$} (m-3-1);

\end{tikzpicture}
\end{center}

\begin{lemma}
\label{lemma G D}
The size of $\theta^*$ is at least $D+1$.
\end{lemma}
\begin{proof}
Denote by $\G_D$ the subset of $\mathcal{M}$ consisting of all
$(D+1)$--th roots of unity. The cardinality of $\G_D$ is $D+1$.
Observe that any $\zeta\in \G_D$ satisfies the condition
$\theta_D(\zeta)=0$ and therefore $\sigma(\zeta)=
(\widetilde{\sigma}\circ\theta_D)(\zeta)$ does not depend on
$\zeta$.

For any $\zeta\in\G_D$, let
$\beta_{\zeta}:\mathcal{M}\rightarrow\mathcal{M}$ be the affine
linear function defined by $\beta_{\zeta}(s):=s+\zeta$. Then by
Theorem \ref{theorem geometrically robust functions}$(iii)$ the maps
$\theta_D'\circ\beta_{\zeta}$ and
$\mu^*\circ\sigma\circ\beta_{\zeta} =\mu^*\circ\widetilde{\sigma}
\circ\theta_D\circ\beta_{\zeta}$ are polynomial and hence
holomorphic. Moreover, $u:=(\mu^*\circ\sigma\circ\beta_{\zeta})(0)$
does not depend on $\zeta$ and consequently
$(\theta_D'\circ\beta_{\zeta})(0)$ depends neither on $\zeta$.

We have $\frac{d}{dt}(\theta_D'\circ\beta_{\zeta})(0)=(D+1)\zeta^D
\sum_{0 \leq k \leq D} \zeta^k X^k$ and therefore the vectors
$\frac{d}{dt}(\theta_D'\circ\beta_{\zeta})(0)$, $\zeta\in\G_D$, are
$\C$--linearly independent. Since the player has a winning strategy,
$\mu^*\circ\sigma\circ\beta_{\zeta}$ is holomorphic and $\omega^*$
is polynomial, we may apply the chain rule to get
\begin{align*}
\frac{d}{dt}(\theta_D'\circ\beta_{\zeta})(0) & =
d\omega^*\big((\mu^*\circ\sigma\circ\beta_{\zeta})(0)\big)\cdot
\frac{d}{dt} (\mu^*\circ\sigma\circ\beta_{\zeta})(0) \\
& = d\omega^*(u)\cdot \frac{d}{dt}
(\mu^*\circ\sigma\circ\beta_{\zeta})(0)
\end{align*}
for any $\zeta\in\G_D$. We conclude that the vectors $\frac{d}{dt}
(\mu^*\circ\sigma\circ\beta_{\zeta})(0)$, $\zeta\in\G_D$, generate a
$\C$--linear space of dimension $D+1$. This implies that the size of
$\mathcal{N}^*$, and hence the size of $\theta^*$, is at least
$D+1$.

The argumentation in the case of the univariate polynomial
interpolation problems given by $(\theta_D,\tau_D'')$ and
$(\theta_D,\tau_D''')$ is almost textually the same. The only
difference is the form of the vectors $\frac{d}{dt}
(\theta_D''\circ\beta_{\zeta})(0) = (D+1)\zeta^D\sum_{1  \leq k \leq
D} k\zeta^k X^{k-1}$ and $\frac{d}{dt}
(\theta_D'''\circ\beta_{\zeta})(0)=(D+1)\zeta^D\sum_{0\leq k \leq D}
\frac{1}{k+1}  \zeta^k X^{k+1}$ for $\zeta\in\G_D$.
\end{proof}

Following \cite[\S 8.1]{Burgisser97}, any polynomial $f$ over $\C$
requires at least $\log\deg f$ essential multiplications to be
evaluated by an ordinary division--free arithmetic circuit. Let
$(\theta_D)_{D\in\N}$ be the sequence of families of univariate
polynomials considered before. Observe that for any
$t\in\mathcal{M}$, the univariate polynomial $\theta_D(t)$ can be
evaluated by an ordinary division--free arithmetic circuit using $4
\lceil\log D-2\rceil$ essential multiplications and that
$\deg\theta_D(t)\leq D$ holds. For all but finitely many
$t\in\mathcal{M}$ we have even $\deg\theta_D(t)=D$. In this sense,
$(\theta_D)_{D\in\N}$ is a sequence of families of univariate
polynomials which are \emph{easy to evaluate}. From Lemma \ref{lemma
G D} we deduce now the following less technical statement.
\begin{theorem}
\label{less technical theorem} There exists a sequence of families
of univariate polynomials which are easy to evaluate such that the
continuous interpolation of these polynomials, or their derivatives,
or their indefinite integrals, requires an amount of arithmetic
operations which is exponential in the number of essential
multiplications of the most efficient ordinary division--free
arithmetic circuits which evaluate these polynomials.
\end{theorem}

Theorem \ref{less technical theorem} extends \cite[Proposition
22]{GiHeMaSo11}, and simplifies its proof.
%
%
\subsection{Learning in neural networks with polynomial activation functions}
\label{subsec: learning in neural networks}
Let us analyze the following architectures of neural networks with
polynomial activation functions. Let $n\in\N$ be a discrete
parameter and $\mathcal{M}:=\A^{n+1}$. For $t\in\A^1$ and
$u=(u_1,\dots,u_n)\in\A^n$, let $\theta(t,u)$ be the polynomial
$$\theta(t,u):=t(u_1X_1+\dots+u_nX_n)^n.$$
Then the formula defining $\theta(t,u)$ describes the architecture
of a two layer network with two neurons, weight vector
$(t,u_1,\dots,u_n)$, inputs $X_1,\dots,X_n$, one output and two
activation functions, namely the polynomials $Y^n$ and $Y$, where
$Y$ is a new indeterminate. This network is evaluable by $2n+2
\lceil \log n \rceil$ arithmetic operations, including $2 \lceil
\log n \rceil$ essential multiplications. Hence the family of
polynomials $\theta$ can be represented by a robust arithmetic
circuit of size $2n+2 \lceil\log n\rceil$. Thus for
$\mathcal{O}:=\text{im\ }\theta$, $\mu:=\text{id\ }\mathcal{M}$ and
$\omega:=\theta$, we have that $\theta=\omega\circ\mu$ and
$\theta:\mathcal{M}\rightarrow\mathcal{O}$ is an abstraction
function associated with the geometrically robust constructible map
$\mu$ and the polynomial map $\omega$. By virtue of Proposition
\ref{propo M}, and the comment that follows it,
$(\theta,\mu,\omega)$ belongs to a compatible collection
$\mathcal{H}$ of abstraction functions.

More precisely, using $\big(4(n+2\lceil\log
n\rceil+1)^2+2\big)(4n+4\lceil\log n\rceil+1)$ arithmetic operations
in $\C$, we may check whether for given weight vectors $(t,u)$ and
$(t',u')$ of $\mathcal{M}$ the identity $\theta(t,u)=\theta(t',u')$
holds for the target functions $\theta(t,u)$ and $\theta(t',u')$ of
the neural networks given by $(t,u)$ and $(t',u')$.

The learning of a target function $\theta(t,u)$ can be interpreted
as the task of interpolating $\theta(t,u)$ in the manner prescribed
by the given network architecture. In view of \cite[Section
3.3.3]{GiHeMaSo11}, the continuous interpolation of the family of
polynomials $\theta$ can be performed, not necessarily efficiently,
using $m\geq 4(n+2\lceil\log n\rceil+1)^2$ random points of $\N^n$
of bit size at most $8 \lceil \log n\rceil + 4$ (see the end of this
subsection).

Since interpolation is a particular case of an exact quiz game
protocol, we are going to ask a more general question, namely
whether for the computation task given by $\theta$,
$\theta':=\theta$ and $\tau:=id_{\mathcal{O}}$ an approximative quiz
game protocol with winning strategy can be efficient. The answer
will be no. Hence our neural networks cannot be learned in a
continuous manner. We confirmed this theoretical conclusion by
computer experiments with the 8.3 (R2014a) Matlab version of the
standard backpropagation algorithm (see \cite[Chapter 6.1]{Hertz91})
which we adapted especially to the case of polynomial activation
functions. It is not worth to reproduce here the particular
experimental results because of the enormous errors they contain. In
particular, if we interpret the steps of the back propagation
algorithm as an infinite sequence of computations over $\C$, the
resulting algorithm cannot converge \emph{uniformly} to an exact
learning process of our neural networks.

Suppose that for the previous computation task there is given an
approximative quiz game protocol with winning strategy. Thus we have
a framed abstract data type carrier $\mathcal{O}^*$, framed data
structures $\mathcal{M}^*$ and $\mathcal{N}^*$, an abstraction
function $\theta^*:\mathcal{M}^*\rightarrow\mathcal{O}^*$ of the
compatible collection $\mathcal{H}$, associated with a geometrically
robust constructible map
$\mu^*:\mathcal{M}^*\rightarrow\mathcal{N}^*$ and a polynomial map
$\omega^*:\mathcal{N}^*\rightarrow\mathcal{O}^*$, and a surjective
geometrically robust constructible map
$\sigma:\mathcal{M}\rightarrow\mathcal{M}^*$ such that by Lemma
\ref{observation} the identities
$\theta=\omega^*\circ\mu^*\circ\sigma = \theta^*\circ\sigma$ and
$\mathcal{O}^*=\mathcal{O}$ hold.
\begin{lemma}
\label{lemma theta star size} The size of $\theta^*$ is at least
$\binom{2n-1}{n-1}\geq 2^{n-1}$.
\end{lemma}
\begin{proof}
The proof is similar to that of Lemma \ref{lemma order}. Observe
that
\[
\theta(t,u)=t\sum_{\stackrel{(\alpha_1,\dots,\alpha_n)\in\Z_{\geq
0}}{\scriptscriptstyle
\alpha_1+\dots+\alpha_n=n}}\frac{n!}{\alpha_1!\dots\alpha_n!}
\,u_1^{\alpha_1}\dots u_n^{\alpha_n} X_1^{\alpha_1}\dots
X_n^{\alpha_n}
\]
holds for any $t\in\A^1$ and $u=(u_1,\dots,u_n)\in\A^n$.

Let $U_1,\dots,U_n$ be new indeterminates and let
$M_n:=\{(\alpha_1,\dots,\alpha_n)\in\Z^n_{\geq
0}:\alpha_1+\dots+\alpha_n=n\}$. Since the $K:=\binom{2n-1}{n-1}$
monomials $U_1^{\alpha_1}\cdots U_n^{\alpha_n}$,
$\underline{\alpha}\in M_n$, are $\C$--linearly independent, we may
conclude that there exists a non--empty Zariski open subset
$\mathcal{U}$ of $\A^{K\times n}$ such that for any point
$(\rho_1,\dots,\rho_K)\in\mathcal{U}$ with
$\rho_i=(\rho_{i1},\dots,\rho_{in})\in\A^n$, $1\leq i \leq K$, the
$(K\times K)$--matrix
\[
{\tt M}_{\rho_1,\dots,\rho_K}:= \Big (
\frac{n!}{\alpha_1!\dots\alpha_n!}\,\rho_{i1}^{\alpha_1}\dots
\rho_{in}^{\alpha_n} \Big )
_{\underline{\alpha}=(\alpha_1,\dots,\alpha_n)\in M_n,\, 1 \leq i
\leq K}
\]
is nonsingular.

Let $\rho=(\rho_1,\dots,\rho_n)\in\A^n$ and let
$\beta_{\rho}:\A^1\rightarrow\mathcal{M}$ be defined by
$\beta_{\rho}(t):=(t,\rho)$. Then the map
$t\mapsto(\theta\circ\beta_{\rho})(t)$ is represented by the
polynomial
\[
(\theta\circ\beta_\rho)(t)=t
\sum_{\stackrel{(\alpha_1,\dots,\alpha_n)\in\Z_{\geq
0}}{\scriptscriptstyle
\alpha_1+\dots+\alpha_n=n}}\frac{n!}{\alpha_1!\dots\alpha_n!}
\rho_1^{\alpha_1}\dots \rho_n^{\alpha_n} X_1^{\alpha_1}\dots
X_n^{\alpha_n},
\]
and is therefore linear in $t$. On the other hand, the map
$\mu^*\circ\sigma\circ\beta_{\rho}:\A^1\to\cN^*$ is geometrically
robust and constructible, and hence polynomial by Theorem
\ref{theorem geometrically robust functions}$(iii)$. Thus
$\mu^*\circ\sigma\circ\beta_{\rho}$ is holomorphic. Since any
$u\in\A^n$ satisfies the condition $\theta(0,u)=0$ we conclude from
Corollary \ref{c:1} that the set $\{ (\mu^*\circ\sigma)(0,u):
u\in\A^n \}$ is finite. Hence the set
$\{(\mu^*\circ\sigma\circ\beta_{\rho})(0):\rho\in\A^n \}$ is also
finite. Therefore there exists a non--empty Zariski open subset
$\Gamma$ of $\A^n$ such that for any point $\rho\in\Gamma$, the
value $(\mu^*\circ\sigma\circ\beta_{\rho})(0)$ is the same, say $v$.
Observe now that the set $\mathcal{U}\cap \Gamma^K$ is non--empty,
because $\mathcal{U}$ and $\Gamma^K$ are non--empty Zariski open
subsets of $\A^{K\times n}$. Hence there exist elements
$\rho_1,\dots,\rho_K$ of $\Gamma$ such that the complex $(K\times
K)$--matrix ${\tt M}_{\rho_1,\dots,\rho_K}$ is nonsingular.

Therefore the vectors $\frac{d}{dt}
(\theta\circ\beta_{\rho_1})(0),\dots,\frac{d}{dt}
(\theta\circ\beta_{\rho_K})(0)$ are $\C$--linearly independent. The
assumption that the player has a winning strategy for the given
approximative quiz game implies
$$\theta\circ\beta_{\rho_j}=\omega^*\circ\mu^*\circ\sigma\circ\beta_{\rho_j}$$
for any $1 \leq j \leq K$. Since
$\mu^*\circ\sigma\circ\beta_{\rho_j}$ is holomorphic for $1 \leq j
\leq K$, we may apply the chain rule to this situation to conclude
\begin{align*}
\frac{d}{dt}(\theta\circ\beta_{\rho_j})(0) &
=d \omega^*\big((\mu^*\circ\sigma\circ\beta_{\rho_j})(0)\big)\cdot
\frac{d}{dt}(\mu^*\circ\sigma\circ\beta_{\rho_j})(0) \\
& = d\omega^*(v)\cdot \frac{d}{dt}
(\mu^*\circ\sigma\circ\beta_{\rho_j})(0)
\end{align*}
for $1 \leq j \leq K$. We infer that the vectors
$\frac{d}{dt}(\mu^*\circ\sigma\circ\beta_{\rho_1})(0),
\dots,\frac{d}{dt}(\mu^*\circ\sigma\circ\beta_{K})(0)$ generate a
$\C$--linear space of dimension $K$. This implies that the size of
$\mathcal{N}^*$, and hence the size of $\theta^*$, is at least
$K=\binom{2n-1}{n-1} \geq 2^{n-1}$.
\end{proof}

Suppose now that it is possible to learn continuously for $n>3$ the
neural networks corresponding to our network architecture. Then
there exists a geometrically robust constructible map
$\widetilde{\sigma}:\mathcal{O}\rightarrow\mathcal{M}$ such that the
identity $\theta=\theta\circ\widetilde{\sigma}\circ\theta$ holds.
This contradicts Lemma \ref{lemma theta star size} with
$\mathcal{O}^*:=\mathcal{O}$, $\mathcal{M}^*:=
\text{im}\left(\widetilde{\sigma}\circ \theta\right)$,
$\theta^*:=\theta\mid_{{\mathcal M}^*}$,
$\mu^*:=\text{id}\mid_{\mathcal{M}^*}=\mu\mid_{\mathcal{M}^*}$,
${\mathcal N}^* := {\mathcal M}^*$,
$\omega^*:=\theta\mid_{\mathcal{N}^*}=\omega\mid_{\mathcal{N}^*}$
and $\sigma:=\widetilde{\sigma}\circ\theta$, because in this case
the size of $\theta^*$ is at most $n+1$, which is strictly smaller
than $2^{n-1}$. Lemma \ref{lemma theta star size} implies therefore
the following learning result for neural network architectures.
\begin{theorem}\label{neural-networks:teor}
Let $n\in\N$ be a discrete parameter. There exists a two--layer
neural network architecture with two neurons, $n$ inputs, one output
and two polynomial activation functions which can be evaluated using
$2 \lceil \log n \rceil$ essential multiplications, such that there
is no continuous algorithm able to learn exactly the corresponding
neural networks.
\end{theorem}

To conclude, we discuss the question how many random points are
sufficient to interpolate the target functions of a general neural
network architecture with polynomial activation functions.

Let be given a neural network architecture with $n$ inputs
$X_1,\dots,X_n$ and $K$ neurons and let $r$ be the length of the
corresponding weight vector. We suppose that all activation
functions are given by univariate polynomials which can be evaluated
by ordinary division--free arithmetic circuits with parameters in
$\R$ using at most $L$ essential multiplications. Thus any weight
vector of $\A^r$ defines a (complex) neural network which can be
evaluated using at most $KL$ essential multiplications. Denote for
any $u\in\A^r$ the target function of the corresponding neural
network by $\theta(u)\in\C[X_1,\dots,X_n]$. Let $m\geq 4(KL+n+1)^2
+2$ and chose $m$ random points $\gamma_1,\dots,\gamma_m\in\N^n$ of
bit size at most $4(KL+1)$. Then $\mathcal{M}:=\{
\theta(u)(\gamma_1),\dots,\theta(u)(\gamma_m) : u\in\A^r \}$ is a
constructible subset of $\A^m$. From \cite[Section
3.3.3]{GiHeMaSo11} we deduce that the map which assigns to each
element $(\theta(u)(\gamma_1),\dots,\theta(u)(\gamma_m))$ of $\cM$
the coefficient vector of the polynomial $\theta(u)$ is
geometrically robust and constructible (over $\R$). This map solves
in continuous (but not necessarily efficient) manner the
interpolation problem given by $\gamma_1,\dots,\gamma_m$ of the
family of polynomial functions $\theta$.
\subsection{Elimination}
Finally we are going to analyze two series of examples of geometric
elimination problems from the point of view of approximative quiz
games. The first one is concerned with a parameterized family of
elimination problems defined on a fixed hypercube and the second one
is related to the computation of characteristic polynomials in
linear algebra.
%
%
\subsubsection{A parameterized family of projections
of a hypercube}
Let $n\in\N$ be a discrete parameter and let
$\mathcal{M}:=\A^{n+1}$. Let $X_1,\dots,X_n,Y$ be indeterminates
over $\C$. For $t\in\A^1$ and $u=(u_1,\dots,u_n)\in\A^n$, let
$\theta(t,u)$ and $\theta'(t,u)$ be the following polynomials:
\begin{align*}
\theta(t,u)&:=\sum_{1 \leq i\leq n} 2^{i-1} X_i +t \prod_{1 \leq
i\leq n} (1+ (u_i-1)X_i),\\
\theta'(t,u)&:=\prod_{\epsilon\in\{ 0,1
\}^n}\big(Y-\theta(t,u)(\epsilon)\big) =\prod_{0 \leq j< 2^n}
\bigg(Y-\Big(j+t \prod_{1 \leq i\leq n} u_i^{[j]_i}\Big)\bigg),
\end{align*}
where $[j]_i$ denotes the $i$--th digit of the binary representation
of the integer $j$ for $0\leq j< 2^n$ and $1 \leq i \leq n$. Let
$\mathcal{O}:=\text{im\ }\theta$, $\mathcal{O}':=\text{im\
}\theta'$, $\mathcal{N}:=\mathcal{M}$,
$\mu:=\text{id}_{\mathcal{M}}$, $\omega:=\theta$ and observe that
the family of polynomials $\theta$ is evaluable by a robust
arithmetic circuit of size $5n$. We may therefore deduce from
Proposition \ref{propo M}, and the comment following it, that
$(\theta,\mu,\omega)$ belongs to a suitable compatible collection
$\mathcal{H}$ of abstraction functions. We suppose that $\theta'$ is
an abstraction function associated with a geometrically robust
constructible map $\mu':\mathcal{M}\to\mathcal{N}'$ and a polynomial
map $\omega:\mathcal{N}'\to\mathcal{O}'$, such that
$(\theta',\mu',\omega')$ belongs to the compatible collection
$\mathcal{H}$.

Observe that for $t\in\A^1$ and $u\in\A^n$, the formulas $$(\exists
X_1)\dots(\exists X_n) (X_1^2-X_1=0 \wedge \dots \wedge X_n^2-X_n=0
\wedge Y- \theta(t,u)(X_1,\dots, X_n)=0 )$$ and $ \theta'(t,u)(Y)=0$
are logically equivalent. In fact,
$\theta'(t,u)=\prod_{\epsilon\in\{ 0,1
\}^n}(Y-\theta(t,u)(\epsilon))$ is the elimination polynomial of the
projection of the hypercube $\{ 0,1 \}^n$ along $\theta(t,u)$.
Therefore there exists a geometrically robust constructible map
$\tau:\mathcal{O}\to\mathcal{O}'$ such that
$\tau\circ\theta=\theta'$ holds.

Suppose now that for the computation task determined by $\theta$ and
$\tau$ there is given a winning strategy of the approximative  quiz
game protocol. Thus we have a framed abstract data type carrier
$\mathcal{O}^*$, framed data structures $\mathcal{M}^*$ and
$\mathcal{N}^*$, an abstraction function $\theta^*$ of the
compatible collection $\mathcal{H}$, associated with a geometrically
robust constructible map
$\mu^*:\mathcal{M}^*\rightarrow\mathcal{N}^*$ and a polynomial map
$\omega:\mathcal{N}^*\rightarrow\mathcal{O}^*$, and a surjective
geometrically robust constructible map
$\sigma:\mathcal{M}\to\mathcal{M}^*$ such that by Lemma
\ref{observation} the identities
$\theta'=\omega^*\circ\mu^*\circ\sigma=\theta^*\circ\sigma$ and
$\mathcal{O}'=\mathcal{O}^*$ hold.

\begin{lemma}
\label{lemma size of theta}
The size of $\theta^*$ is at least $2^n$.
\end{lemma}
\begin{proof}
Let $T,U_1,\dots,U_n$ be new indeterminates, $U:=(U_1,\dots,U_n)$
and let $F:=\prod_{0\le j<2^n} (Y-(j + T \prod_{1\leq i\leq n}
U_i^{[j]_i} ))$. We write $F=Y^{2^n} + B_1 Y^{2^n -1} +\dots+
B_{2^n}$ with $B_k\in\C[T,U]$ for $1\leq k\leq 2^n$. Arguing as in
\cite[\S 4.2]{CaGiHeMaPa03} we deduce that for $1 \leq k \leq 2^n$,
the coefficient $B_k$ is of the form
\[
B_k=  (-1)^k \sum_{0\leq j_1<\dots<j_k<2^n} j_1\dots j_k + T L_k +
\text{\ terms\ of\ higher\ degree\ in\ $T$,}
\]
where $L_1,\dots,L_{2^n} \in\C[U]=\C[U_1,\dots,U_n]$ are
$\C$--linearly independent. Therefore there exists a non--empty
Zariski open subset $\mathcal{U}$ of $\A^{2^n\times n}$ such that
for any point $(u_1,\dots,u_{2^n})\in\mathcal{U}$ with
$u_1,\dots,u_{2^n}\in\A^n$, the $(2^n\times 2^n)$--matrix
\begin{equation}
{\tt M}_{L_1\klk L_{2^n}}:=\big(L_k(u_l)\big)_{1\leq k,l \leq 2^n}
\label{matrix}
\end{equation}
is nonsingular.

Since any $u\in\A^n$ satisfies the condition
$\theta(0,u)=\sum_{1\leq i\leq n} 2^{i-1} X_i$ and $\theta(0,u)$ is
therefore constant, we conclude from Corollary \ref{c:1} that the
set $\{(\mu^*\circ\sigma)(0,u):u\in\A^n \}$ is finite. Hence there
exists a non--empty Zariski open subset $\Gamma$ of $\A^n$ such that
for any point $u\in\Gamma$, the value $(\mu^*\circ\sigma)(0,u)$ is
the same, say $v$. On the other hand $\mathcal{U}\cap \Gamma^{2^n}$
is nonempty. Therefore there exist elements
$u_1,\dots,u_{2^n}\in\Gamma$ such that the $(2^n\times 2^n)$--matrix
${\tt M}_{L_1\klk L_{2^n}}$ of \eqref{matrix} is non--singular.

The maps $\mu^*\circ\sigma:\A^{n+1}\to\cN^*$ and
$\theta'(t,u):\A^{n+1}\to\cO^*$ are geometrically robust and
constructible, and hence polynomial by Theorem \ref{theorem
geometrically robust functions}$(iii)$. For $1\leq k\leq 2^n$ and
$t\in\A^1$, let $\varepsilon_k(t):=(\mu^*\circ\sigma)(t,u_k)$ and
$\delta_k(t):=\theta'(t,u_k)$. Observe that $\varepsilon_k$ and
$\delta_k$ are polynomial maps with domain of definition $\A^1$ and
therefore holomorphic. Moreover we have
$\varepsilon_k(0)=(\mu^*\circ\sigma)(0,u_k)=v$ for $1\leq k\leq
2^n$. Since $\theta'=\omega^*\circ\mu^*\circ\sigma$, we deduce
$\delta_k=\omega^*\circ\varepsilon_k$. To this composition we may
apply the chain rule to conclude that
\[
\frac{d\delta_k}{dt}(0)=d\omega^*(\varepsilon_k(0))\cdot
\frac{d\varepsilon_k}{dt}(0)=d\omega^*(v)\cdot\frac{d\varepsilon_k}{dt}(0).
\]

Observe now that $L_1(u),\dots,L_{2^n}(u)$ are the coefficients of
the polynomial $\frac{\partial F }{\partial T}(0,u)$. This implies
that
$$d\omega^*(v)\cdot \frac{d\varepsilon_k}{dt}(0)=
\frac{d\delta_k}{dt}(0)=\frac{\partial\theta'}{\partial t} (0,u_k)=
\big(L_1(u_k),\dots,L_{2^n}(u_k)\big)$$
for $1\leq k\leq 2^n$. From the non--singularity of the $(2^n\times
2^n)$--matrix ${\tt M}_{L_1\klk L_{2^n}}$ of \eqref{matrix} we
deduce that the vectors $\frac{d\varepsilon_1}{dt}(0), \dots,
\frac{d\varepsilon_{2^n}}{dt}(0)$ generate a $\C$--linear space of
dimension $2^n$. This implies that the size of $\mathcal{N}^*$, and
hence the size of $\theta^*$, is at least $2^n$.
\end{proof}

We are now going to elaborate an important aspect of this family of
examples. Observe that any polynomial $\theta(t,u)$ can be evaluated
by a division--free arithmetic circuit using $n-1$ essential
multiplications. From \cite[Section 3.3.3]{GiHeMaSo11} we deduce the
following facts: \\
there exist $K:=16n^2 + 2$ points $\xi_1,\dots,\xi_K \in\Z^n$ of bit
length at most $4n$ such that for any two polynomials
$f,g\in\ol{\mathcal{O}}$, the equalities $f(\xi_1)=g(\xi_1),\dots,
f(\xi_K)=g(\xi_K)$ imply $f=g$. Thus the polynomial map
$\tilde{\sigma}:\mathcal{O}\to\A^K$ defined by
$\tilde{\sigma}(f):=(f(\xi_1),\dots,f(\xi_K))$ is injective.
Moreover $\mathcal{M}^*:=\tilde{\sigma}(\mathcal{O})$ is an
irreducible constructible subset of $\A^K$. Finally the
constructible map $\Phi:=\tilde{\sigma}^{-1}$ which maps
$\mathcal{M}^*$ onto $\mathcal{O}$ is geometrically robust.

For $\epsilon\in\{0,1\}^n$, we denote by
$\Phi_{\epsilon}:\mathcal{M}^*\to\A^1$ the map
$\Phi_{\epsilon}(v):=\Phi(v)(\epsilon)$. One sees easily that
$\Phi_{\epsilon}$ is a geometrically robust constructible function.
Observe that the identities
$$\Phi_{\epsilon}(\tilde{\sigma}(\theta(t,u)))=\Phi(\tilde{\sigma}(\theta(t,u))(\epsilon))
=(\tilde{\sigma}^{-1}\circ\tilde{\sigma})(\theta(t,u))(\epsilon)=\theta(t,u)(\epsilon)$$
hold for any $t\in\A^1$ and $u\in\A^n$.

For a given point $v\in\mathcal{M}^*$, let
$$P(v,Y):=\prod_{\epsilon\in \{ 0,1 \}^n} (Y-\Phi_{\epsilon}(v)):=
\prod_{\epsilon\in \{ 0,1 \}^n} (Y-\Phi(v)(\epsilon))$$
and let $\mathcal{O}^*:=\{ P(v,Y):v\in\mathcal{M}^* \}$. Then the
coefficients of $P$ with respect to $Y$ are geometrically robust
constructible functions with domain of definition $\mathcal{M}^*$
and $\mathcal{O}^*$ is a framed abstract data type carrier. Hence
the map $\theta^*:\mathcal{M}^*\to\mathcal{O}^*$,
$\theta^*(v):=P(v,Y)$ is geometrically robust and constructible.
Suppose now that there is given a framed data structure
$\mathcal{N}^*$ and that $\theta^*$ is an abstraction function
associated with a geometrically robust constructible map
$\mu^*:\mathcal{M}^*\to\mathcal{N}^*$ and a polynomial map
$\omega^*:\mathcal{N}^*\to\mathcal{O}^*$, such that
$(\theta^*,\mu^*,\omega^*)$ belongs to the compatible collection
$\mathcal{H}$.

From the identities
$$P(\tilde{\sigma}(\theta(t,u)), Y)=\! \prod_{\epsilon\in\{ 0,1 \}^n}\!(Y
- \Phi_{\epsilon}(\tilde{\sigma}(\theta(t,u))))
=\!\prod_{\epsilon\in\{ 0,1 \}^n} \!(Y-\theta(t,u)(\epsilon))=
\theta'(t,u)$$
we deduce that $\omega^*\circ\mu^*\circ\tilde{\sigma}\circ\theta=
\theta^*\circ\tilde{\sigma}\circ\theta=\theta'$ holds. This means
that $\tilde{\sigma}\circ\theta$ and $\mu^*$ define a winning
strategy for the exact quiz game protocol of the computation task
given by $\theta$ and $\tau$. This implies by Lemma \ref{lemma size
of theta} that the size of $\mathcal{N}^*$ is at least $2^n$. In
other words, any representation of the polynomial $P$ has size
exponential in $n$.

We argue now that the polynomial $P$ is a natural elimination
object. Let $\Theta:=\sum_{1\leq i\leq n} 2^{i-1}X_i + T\prod_{1\leq
i\leq n}(1 + (U_i-1)X_i)$ and let $V_1,\dots,V_k$ be new
indeterminates. Observe that the existential first--order formula
\begin{align}
&(\exists X_1)\dots(\exists X_n)(\exists T)(\exists
U_1)\dots(\exists U_n)\nonumber\\  &\bigg(\bigwedge_{1\leq i\leq
n}X_i^2-X_i=0 \wedge \bigwedge_{1\leq k\leq K} V_k=\Theta(T,U,\xi_k)
\wedge Y=\Theta(T,U,X) \bigg) \label{formula}
\end{align}
describes the constructible subset
$\{(v,y)\in\A^{K+1}:v\in\mathcal{M}^*,y\in\C, P(v,y)=0 \}$ of
$\A^{K+1}$. Interpreting the coefficients of $P$ as elements of the
function field $\C(\ol{\mathcal{M}}^*)$ of the irreducible algebraic
variety $\ol{\mathcal{M}}^*$, one sees easily that $P$ is the
greatest common divisor in $\C(\ol{\mathcal{M}}^*)[Y]$ of all
polynomials of $\C[\ol{\mathcal{M}}^*][Y]$ which vanish identically
on the constructible subset of $\A^{K+1}$ defined by formula
\eqref{formula}. Hence $P$ is an elimination polynomial
parameterized by $\mathcal{M}^*$.

Observe that the polynomials contained in the formula
\eqref{formula} can be represented by a robust arithmetic circuit of
size $O(n^3)$. Therefore the formula \eqref{formula} is also of size
$O(n^3)$. Thus we have proved the following statement.
\begin{theorem}\label{elimination-exponential-size:teor}
Let $n\in\N$ be a discrete parameter. There exists a first--order
formula of size $O(n^3)$ of the elementary theory of $\C$
determining a framed data structure $\mathcal{M}^*$ of size $O(n^2)$
and a family $P$ of univariate elimination polynomials parameterized
by $\mathcal{M}^*$ such that the following holds. Let
$\mathcal{O}^*$ be the framed abstract data type carrier defined by
the family of polynomials $P$ and
$\theta^*:\mathcal{M}^*\rightarrow\mathcal{O}^*$ the encoding of
$\mathcal{O}^*$ given by the parameterization of $P$. Then
$\theta^*$ is a surjective geometrically robust constructible map.
Moreover, for any framed data structure $\mathcal{N}^*$ and any
geometrically robust constructible map $\mu^*:\mathcal{M}^*
\rightarrow \mathcal{N}^*$ and any polynomial map
$\omega^*:\mathcal{N}^* \rightarrow \mathcal{O}^*$, such that
$(\theta^*,\mu^*,\omega^*)$ determines an abstraction function
associated with $\mu^*$ and $\omega^*$, the size of $\mathcal{N}^*$
is at least $2^n$.
\end{theorem}

In conclusion, polynomial size formulas may produce elimination
polynomials of exponential size for any (reasonable) representation.
This conclusion is essentially the content of \cite[Theorem
4]{CaGiHeMaPa03} (see also \cite[Theorem 5]{GiHe01}, \cite[Theorem
15]{HeKuRo} and \cite[Theorem 8]{HKRPSantalo}).

The leading idea of the proof of Theorem
\ref{elimination-exponential-size:teor} was to use an existential
formula, namely \eqref{formula}, to encode a suitable interpolation
problem whose solution is the interpolation polynomial
$P\in\C(\ol{\mathcal{M}}^*)[Y]$. This encoding relates elimination
with interpolation in the sense of Section \ref{subsubsec:
multivariate polynomial interpolation} and explains the similarity
of the proofs of Theorem \ref{theorem lower bound} and Theorem
\ref{elimination-exponential-size:teor}. In fact, the example
exhibited in the proof of Theorem
\ref{elimination-exponential-size:teor} entails nothing but a
hardness result for a suitable interpolation of a polynomial family,
given by $\Theta$, on the constructible set $\mathcal{M}^*$. Using
formula \eqref{formula} we obtain finally our hardness result for
elimination, namely Theorem \ref{elimination-exponential-size:teor}.
%
%
\subsubsection{Quiz games in elementary linear algebra: Characteristic polynomials}
For every positive integer $n\in\N$, we first consider the ring
$(\mathbf{M}_n(\C), +, \times)$ of all square $n\times n$ matrices
with complex entries, with their usual addition and multiplication
operations. Let $\mathfrak{R}$ be the disjoint union of all these
rings.
\[
\mathfrak{R}:= \bigcup_{n\in \N} {\boldface{M}}_n(\C)
\]
A framed abstract data type carrier of \emph{square matrices} is now
a constructible subset of some $\mathbf{M}_n(\C)$, $n\in\N$,
contained in $\mathfrak{R}$. We are now going to introduce  two
binary elementary operations on $\mathfrak{R}$, namely the Kronecker
sum ($\oplus$) and Kronecker product ($\otimes$). However,
$(\mathfrak{R},\oplus, \otimes)$ will not be a ring, because
Kronecker sum and product are not distributive.

Recall that for two square matrices $A:=(a_{i,j})_{1\leq i,j \leq
n}\in \mathbf{M}_n(\C)$ and $B:=(b_{k,\ell})_{1\leq k,\ell \leq
m}\in \mathbf{M}_m(\C)$, seen as elements of $\mathfrak{R}$, their
\emph{Kronecker product} $A\otimes B$ is the square $(mn\times
mn)$--matrix given by the following rule:
\[ A\otimes B:=
\left(\begin{matrix}
 a_{1,1}B & a_{1,2}B & \cdots & a_{1,n}B\\
 a_{2,1}B & a_{2,2}B & \cdots & a_{2,n}B\\
 \vdots   &          &        & \vdots  \\
 a_{n,1}B & a_{m,2}B & \cdots & a_{n,n}B
\end{matrix}\right)\in \boldface{M}_{nm}(\C).
\]
Here $a_{i,j}B\in\mathbf{M}_m(\C)$ is the matrix obtained by
multiplying each entry of the matrix $B$ by the scalar $a_{i,j}$.
The \emph{Kronecker sum} $A \oplus B$ of the matrices $A$ and $B$ is
the square $(mn\times mn)$--matrix defined in terms of the Kronecker
product as follows:
\[
A\oplus B:= A \otimes \mathrm{Id}_m + \mathrm{Id}_n\otimes B\in
\mathbf{M}_{nm}(\C).
\]
Here $\mathrm{Id}_n$ and $\mathrm{Id}_m$ denote the $n\times n$ and
$m\times m$ identity matrices and $+$ is the usual addition of $nm$
square matrices.

Using the Kronecker sum and product of square matrices we are now
going to introduce a suitable abstraction function. Let $k$ be a
non-negative integer and $n:=2^k$. The image of the abstraction
function we are going to define will be a constructible subset of
the $\C$--vector space $\mathbf{M}_n(\C)$ contained in
$\mathfrak{R}$. Let us consider the framed data structure
$\mathcal{M}:=\A^{k+1}$. We represent the elements of $\mathcal{M}$
by $(s, u_1,\ldots, u_k)\in \A^{k+1}$. Let ${\mathcal N}:= \C\times
(\mathbf{M}_2(\C))^k$ and observe that it is a $\C$--vector space of
dimension $1+4k=4\log n+1$ and hence constructible.

Consider the mapping
\begin{align*}
\mu:\quad{\mathcal{M}}&\longrightarrow \mathcal N\\
 (s,u_1,\ldots, u_k) & \longmapsto  \left(s,\left(\begin{matrix}1 & 0 \\
0 & u_1\end{matrix}\right), \ldots, \left(\begin{matrix}1 & 0 \\ 0 &
u_k\end{matrix}\right)  \right),
\end{align*}
which is polynomial and hence geometrically robust and
constructible. For $s\in\C$ and $A_1,\dots,A_k\in \mathbf{M}_2(\C)$,
we define a polynomial mapping $\omega:{\mathcal N}\to
\mathbf{M}_n(\C)$ by means of the Kronecker sum and product and
$\C$--linear operations in $\mathbf{M}_n(\C)$ as follows:
\[
\omega(s, A_1, \ldots, A_k):= \left(
\bigoplus_{i=1}^{k}\left(\begin{matrix} 0 & 0 \\0 & 2^{k-i}
\end{matrix}\right)\right)+ s A_1\otimes \cdots \otimes A_k.
\]
Finally, let $\theta:=\omega\circ\mu$ and $\mathcal{O}$ be the image
of $\theta$. Then $\mathcal{O}$ is a constructible subset of
${\boldface{M}}_{n}(\C)$ contained in $\mathfrak{R}$. Moreover,
$\mu:\cM\to\cN$, $\omega:\cN\to\cO$ and hence $\theta:\cM\to\cO$ are
polynomial, the surjective mapping $\theta$ being an abstraction
function associated with $\mu$ and $\omega$. The abstraction
function $\theta:\mathcal{M}\longrightarrow {\mathcal O}$ can be
made explicit by means of the following identity:
\[
\theta(s,u_1,\dots,u_k)=\left(
\bigoplus_{i=1}^{k}\left(\begin{matrix} 0 & 0 \\0 & 2^{k-i}
\end{matrix}\right)\right)+ s \left(\begin{matrix}1 & 0 \\ 0 &
u_k\end{matrix}\right)\otimes \cdots \otimes \left(\begin{matrix}1 &
0 \\ 0 & u_1\end{matrix}\right).
\]
This means that $\theta$ can be expressed in terms of the Kronecker
sum and product using only $2k$ matrix and Kronecker operations.

Let $\chi$ be the map which associates to each element of
$\mathfrak{R}$ its characteristic polynomial. Notice that the
restriction of $\chi$ to $\mathbf{M}_n(\C)$ is a polynomial map. Let
$\mathcal{O}':=\chi(\mathcal{O})$, $\mathcal{N}':=\mathcal{M}$,
$\mu':=\text{id}_{\mathcal{M}}$, $\omega':=\chi\circ\theta$. Then
$\mathcal{O}'$ is a constructible subset of the $\C$--vector space
of univariate polynomials of degree at most $n$ and $\omega'$ is
polynomial. Hence $\mathcal{O}'$ is a framed abstract data type
carrier and $\theta':\mathcal{M}\to \mathcal{O}'$ is an abstraction
function associated with $\mu'$ and $\omega'$.
We have therefore the following commutative diagram of polynomial
maps:
\begin{center}
\begin{tikzpicture}[scale=0.5]


\node[ shape=rectangle] (v5) at (12,0)   {$\mathcal{N}$};
\node[ shape=rectangle] (v6) at (16,0)   {$\mathcal{M}$};
\node[ shape=rectangle] (v7) at (20,0)   {$\mathcal{N}'$};

\node[ shape=rectangle] (v8) at (16,4)   {$\mathcal{O}$};
\node[ shape=rectangle] (v9) at (20,4)   {$\mathcal{O}'$};

\path[color=black,->,font=\scriptsize,>=angle 90]

(v5) edge[->] node[above] {$\omega$}  (v8)
(v6) edge[->] node[below] {$\mu$}  (v5)
(v6) edge[->] node[left] {$\theta$}  (v8)

(v6) edge[->] node[left] {$\theta'$}  (v9)

(v6) edge[->] node[below] {$\mu'$}  (v7)
(v8) edge[->] node[above] {$\chi$}  (v9);


\path[color=black,->,font=\scriptsize,>=angle 90]
(v7) edge[->] node[right] {$\omega'$}  (v9);

\end{tikzpicture}
\end{center}

\begin{lemma} With these notations, the following properties hold:
\[ \scalebox{0.9}[0.9]{$\displaystyle
\left( \bigoplus_{i=1}^{k}\left(\begin{matrix} 0 & 0 \\0 & 2 \end{matrix}\right)^{k-i}\right) =
\left(\begin{matrix} 0 & 0 \\0 & 2^{k-1} \end{matrix}\right)
\oplus \cdots \oplus
\left(\begin{matrix} 0 & 0 \\0 & 2^{1} \end{matrix}\right)
\oplus
\left(\begin{matrix} 0 & 0 \\0 & 1 \end{matrix}\right)
= \left(\begin{matrix} 0 & 0 & \cdots & 0 \\
0 & 1 & \cdots & 0\\
0 & 0 & \cdots & \vdots\\
0 & 0 & \cdots & n-1\end{matrix}\right).$}
\]
Additionally,
\[
\left(\begin{matrix}1 & 0 \\ 0 & u_k\end{matrix}\right)\otimes
\cdots \otimes \left(\begin{matrix}1 & 0 \\ 0 &
u_1\end{matrix}\right)= {\rm Diag}\left(\prod_{i=1}^{k}u_i^{[j]_i}\;
:\; 0\leq j \leq n-1\right),
\]
where $\mathrm{Diag}$ denotes the diagonal matrix. Moreover, we have
\[
\theta'(s,u)=\chi_{\theta(s,u_1,\dots,u_n)}(Y)=\prod_{0 \leq j\leq
n-1} \bigg(Y-\Big(j+s \prod_{1 \leq i\leq k}
u_i^{[j]_i}\Big)\bigg)\in {\mathcal O}'.
\]
\end{lemma}

\begin{proof}
The first equality holds by induction, observing that
\[
\left( \bigoplus_{i=1}^{k}\left(\begin{matrix} 0 & 0 \\0 & 2
\end{matrix}\right)^{k-i}\right) = \left(\begin{matrix} 0 & 0 \\0 &
2^{k-1} \end{matrix}\right)\oplus \cdots \oplus \left(\begin{matrix}
0 & 0 \\0 & 2^{1} \end{matrix}\right)\oplus \left(\begin{matrix} 0 &
0 \\0 & 1 \end{matrix}\right)
\]
holds. The second equality also follows by induction. As for the
third one, it just follows combining the previous equalities with
the definition of $\theta$ and the definition of the characteristic
polynomial.
\end{proof}

Suppose now that for the computational task determined by $\theta$
and $\chi$ there is given a winning strategy of the approximative
quiz game protocol. Thus we have a framed abstract data type carrier
$\mathcal{O}^*$, framed data structures $\mathcal{M}^*$ and
$\mathcal{N}^*$, an abstraction function $\theta^*$ of the
compatible collection $\mathcal{H}$ associated with a geometrically
robust constructible map $\mu^*:\mathcal{M}^* \to \mathcal{N}^*$ and
a polynomial map $\omega^*:\mathcal{N}^* \to \mathcal{O}^*$, and a
surjective  geometrically robust constructible map
$\sigma:\mathcal{M}\to\mathcal{M}^*$, such that by Lemma
\ref{observation} the identities
$\theta'=\omega^*\circ\mu^*\circ\sigma=\theta^*\circ\sigma$ and
$\mathcal{O}'=\mathcal{O}^*$ hold. This situation is depicted by the
following commutative diagram:

\begin{center}
\begin{tikzpicture}[scale=0.5]
\node[ shape=rectangle] (v1) at (0,3)   {$\mathcal{M}$};
\node[ shape=rectangle] (v2) at (3,3)   {{${\mathcal O}$}};
\node[ shape=rectangle] (v3) at (3,0)   {$\mathcal{O}'$};
\node[ shape=rectangle] (v4) at (0,0)   { {${\mathcal N}'$}};
\node[ shape=rectangle] (v5) at (0,6)   {{${\mathcal N}$}};
\node[ shape=rectangle] (v6) at (9,3)   {$\mathcal{M}^*$};%
\node[ shape=rectangle] (v7) at (9,0)   {{${\mathcal N}^*$}};%
\node[ shape=rectangle] (v8) at (6,0)   {{${\mathcal O}^*$}};%
\node[ shape=rectangle] (v9) at (4.5,0)   {$=$};

\path[color=black,->,font=\scriptsize,>=angle 90]
(v6) edge[->] node[right] {$\mu^*$}  (v7)
(v7) edge[->] node[below] {$\omega^*$}  (v8)
(v6) edge[->] node[left] {$\theta^*$}  (v8)

;

\path[color=black,->,font=\scriptsize,>=angle 90]
(v1) edge[->] node[left]{$\theta'$}  (v3)
(v1) edge[->] node[left] {$\mu'$}  (v4)
(v4) edge[->] node[below] {$\omega'$}  (v3)
;

\path[color=black,->,font=\scriptsize,>=angle 90]
(v1) edge[->] node[above] {$\theta$}  (v2)
(v1) edge[->] node[left] {$\mu$}  (v5)
(v5) edge[->] node[right] {$\omega$}  (v2)

;

\path[color=black,->,font=\scriptsize,>=angle 90]
(v2) edge[->] node[right] {$\chi$}  (v3)
(v1) edge[->, bend left] node[above right] {$\sigma$}  (v6)

;

\end{tikzpicture}
\end{center}

By similar arguments as in the proof of Lemma \ref{lemma size of
theta} we obtain the following statement.
\begin{lemma}
The size of $\theta^*$ is at least $n=2^k$.
\end{lemma}

This yields the following result, which may be seen as a linear
algebra avatar of Theorem \ref{elimination-exponential-size:teor}.
\begin{theorem}\label{characteristic-polynomial-exponential-size:teor}
Let $k\in\N$ be a discrete parameter and let $n=2^k$. There exists a
term of size $O(k)$ in $k+1$ variables in the first--order language
of $(\mathfrak{R}, \oplus, \otimes)$ which involves also
$\C$--linear operations in $\mathbf{M}_n(\C)$ and has the following
properties.

Let $\mathcal{M}:=\A^{k+1}$. The term describes a polynomial mapping
from $\mathcal{M}$ onto a constructible subset $\mathcal{O}$ of
$\mathbf{M}_n(\C)$ and represents therefore an abstraction function
$\theta:\mathcal{M}\to\mathcal{O}$. Let $\mathcal{O}'$ be the set of
characteristic polynomials of the elements of $\mathcal{O}$ and
$\chi:\mathcal{O}\to\mathcal{O}'$ the corresponding polynomial map.
Then, by means of $\chi\circ\theta$, the mentioned term determines a
family of univariate monic polynomials of degree $n$ parameterized
by $\mathcal{M}$ and converts $\mathcal{O}'$ into a framed abstract
data type carrier together with a surjective polynomial mapping
$\theta':\mathcal{M}\to\mathcal{O}'$ representing an abstraction
function.

Moreover, any approximative quiz game with winning strategy giving
rise to framed data structures $\mathcal{M}^*$ and $\mathcal{N}^*$,
geometrically robust constructible maps
$\sigma:\mathcal{O}\to\mathcal{M}^*$,
$\mu^*:\mathcal{M}^*\to\mathcal{N}^*$ and
$\theta^*:\mathcal{M}^*\to\mathcal{O}'$, and a polynomial map
$\omega^*:\mathcal{N}^*\to\mathcal{O}'$, such that $\sigma$ is
surjective and the diagram
\begin{center}
\begin{tikzpicture}[scale=0.5]
\node[ shape=rectangle] (v1) at (0,4)   {$\mathcal{M}$};
\node[ shape=rectangle] (v2) at (0,8)   {{${\mathcal O}$}};
\node[ shape=rectangle] (v3) at (4,8)   {$\mathcal{O}'$};
\node[ shape=rectangle] (v4) at (8,4)   { {${\mathcal N}^*$}};
\node[ shape=rectangle] (v5) at (4,4)   {{${\mathcal M}^*$}};

\path[color=black,->,font=\scriptsize,>=angle 90]
(v1) edge[->] node[left] {$\theta$}  (v2)
(v2) edge[->] node[above] {$\chi$}  (v3)
(v1) edge[->] node[below] {$\theta'$}  (v3)
(v2) edge[->] node[above] {$\sigma$}  (v5)
(v4) edge[->] node[right] {$\omega^*$}  (v3)
(v5) edge[->] node[below] {$\mu^*$}  (v4)
(v4) edge[->] node[right] {$\omega^*$}  (v3)
(v5) edge[->] node[right] {$\theta^*$}  (v3)

;

\end{tikzpicture}
\end{center}
commutes, requires that $\mathcal{N}^*$ has size at least $n=2^k$.
\end{theorem}

%
%
\appendix
\section{Further facts on geometrically robust constructible maps}
\label{section: futher facts on geom rob maps}
In the proof of Proposition \ref{p:1} we make use of the following
statement, which is also a key ingredient of the proof of Theorem
\ref{theorem geometrically robust functions}. It is a slight
extension of \cite[Lemma 2]{CaGiHeMaPa03}.
\begin{lemma}
\label{lemma module} Let $\mathcal{M}$ and $\mathcal{N}$ be
constructible subsets of suitable affine spaces over $\C$ and let
$\Phi:\mathcal{M}\rightarrow\mathcal{N}$ be a surjective polynomial
map (thus $\Phi$ induces a structure of
$\C[\ol{\mathcal{N}}]$--module on $\C[\ol{\mathcal{M}}]$).

Suppose that $\Phi$ satisfies the following condition: each sequence
$(x_k)_{k\in\N}$ of elements of $\mathcal{M}$, such that
$(\Phi(x_k))_{k\in\N}$ converges in the Euclidean topology of
$\mathcal{N}$ to a point of $\mathcal{N}$, is bounded. Then for any
point $y\in\mathcal{N}$ with maximal vanishing ideal
$\mathfrak{m}_y$ in $\C[\ol{\mathcal{N}}]$, the
$\C[\ol{\mathcal{N}}]_{\mathfrak{m}_y}$--module
$\C[\ol{\mathcal{M}}]_{\mathfrak{m}_y}$ is finite.
\end{lemma}

Since Lemma \ref{lemma module} is of its own interest, we are going
to give a simple geometric proof of it.

\begin{proof}
The surjective polynomial map
$\Phi:\mathcal{M}\rightarrow\mathcal{N}$ induces embeddings of
$\C[\ol{\mathcal{N}}]$ into $\C[\ol{\mathcal{M}}]$ and, if $\cM$ is
irreducible, of $\C(\ol{\mathcal{N}})$ into $\C(\ol{\mathcal{M}})$.
In this case, we claim that $\C(\ol{\cM})$ is a finite extension of
$\C(\ol{\cN})$. Indeed, let $r\ge 0$ be the transcendence degree of
$\C(\ol{\cM})$ over $\C(\ol{\cN})$. Then there exists a nonempty
Zariski open subset $U$ of $\ol{\cN}$ such that $\dim\Phi^{-1}(y)=r$
for any $y\in U$. Let $y\in U\cap\cN$. If $r>0$, then $\Phi^{-1}(y)$
is unbounded by \cite[Lemma 1]{CaGiHeMaPa03}. In particular, there
exists an unbounded sequence $(x_k)_{k\in\N}$ in $\cM$ with
$\Phi(x_k)=y$, contradicting the assumptions of the lemma. We
conclude that $r=0$, and hence $\C(\ol{\cM})$ is a finite extension
of $\C(\ol{\cN})$.

Let $\mathcal{A}$ be the integral closure of $\C[\ol{\mathcal{N}}]$
in $\C[\ol{\mathcal{M}}]$. If $\mathcal{M}$ is irreducible, as
$\C(\ol{\cM})$ is a finite extension of $\C(\ol{\cN})$, by, e.g.,
\cite[Corollary 13.13]{Eisenbud95}, we deduce that $\mathcal{A}$ is
a finite $\C[\ol{\mathcal{N}}]$--module. Since
$\C[\ol{\mathcal{N}}]$ is a reduced noetherian ring, one concludes
now easily that the same is true in the general case, when
$\mathcal{M}$ is not necessarily irreducible. Therefore there exists
an affine variety $V$ over $\C$ with $\mathcal{A}\cong \C[V]$. The
embeddings of $\mathcal{A}$ in $\C[\ol{\mathcal{M}}]$ and of
$\C[\ol{\mathcal{N}}]$ in $\mathcal{A}$  induce a commutative
diagram of morphisms of affine varieties
\[
\xymatrix{
     &  V \ar[dd]^{\Psi} \\
\ol{\mathcal{M}} \ar[ur]^{\widetilde{\Phi}} \ar[dr]_{\Phi}   &  \\
     & \ol{\mathcal{N}}}
\]
where the image of $\widetilde{\Phi}$, and hence
$\widetilde{\Phi}(\mathcal{M})$, are Zariski dense in $V$ and $\Psi$
is finite.

Consider now an arbitrary point $y\in\mathcal{N}$ with maximal
vanishing ideal $\mathfrak{m}_y$ in $\C[\ol{\mathcal{N}}]$ and let
$z\in V$ an arbitrary point of $V$ with $\Psi(z)=y$. Observe that
such a point $z$ exists because $\Psi$ is a finite morphism.
Moreover, let $\mathfrak{M}_z$ be the maximal vanishing ideal of $z$
in $\C[V]$.
\begin{claim}
\label{claim *} $\widetilde{\Phi}^{-1}(z)$ is nonempty.
\end{claim}
\begin{proof}
Since $\widetilde{\Phi}(\mathcal{M})$ is Zariski dense in $V$, it is
also dense in $V$ with respect to the Euclidean topology of $V$.
Therefore there exists a sequence $(x_k)_{k\in\N}$ of points of
$\mathcal{M}$ such that $(\widetilde{\Phi}(x_k))_{k\in\N}$ converges
to $z$. Since $\Psi$ is continuous with respect to the Euclidean
topologies of $V$ and $\ol{\mathcal{N}}$, the sequence
$(\Phi(x_k))_{k\in\N}= ((\Psi\circ\widetilde{\Phi})(x_k))_{k\in\N}$
converges to $y=\Psi(z)$.

By assumption $\Phi$ satisfies the condition of Lemma \ref{lemma
module}. Therefore the sequence $(x_k)_{k\in\N}$ is bounded. Without
loss of generality we may assume that $(x_k)_{k\in\N}$ converges
with respect to the Euclidean topology of $\ol{\mathcal{M}}$ to a
point $x\in\ol{\mathcal{M}}$. Then the continuity of
$\widetilde{\Phi}$ implies $\widetilde{\Phi}(x)= \lim_{k \to \infty}
\widetilde{\Phi}(x_k) = z$. Hence $\widetilde{\Phi}^{-1}(z)$ is
nonempty.
\end{proof}

\begin{claim}
\label{claim **} $\widetilde{\Phi}^{-1}(z)$ is finite.
\end{claim}

\begin{proof}
Suppose on the contrary that $\widetilde{\Phi}^{-1}(z)$ is infinite.
By \cite[Lemma 1]{CaGiHeMaPa03} there exists an unbounded sequence
$(v_k)_{k\in\N}$ in $\widetilde{\Phi}^{-1}(z)$. Since $\mathcal{M}$
is dense in $\ol{\mathcal{M}}$ in the Euclidean topology, there
exists for each $k\in\N$ a sequence $(x_{k_i})_{i\in\N}$ of points
of $\mathcal{M}$ which converges to $v_k$ in the Euclidean topology.
From the continuity of $\widetilde{\Phi}$ and the fact that
$v_k\in\widetilde{\Phi}^{-1}(z)$, we deduce that the sequence
$(\widetilde{\Phi}(x_{k_i}))_{i\in\N}$ converges to $z$ in the
Euclidean topology. Without loss of generality we may assume $ \|
v_k - x_{k_i}\| < \frac{1}{i}$ and $ \| z -
\widetilde{\Phi}(x_{k_i}) \| < \frac{1}{i}$ for any $k,i\in\N$,
where $\| \cdot \|$ denotes the Euclidean norm of $\ol{\mathcal{M}}$
and $V$, respectively. Therefore we have $ \| v_k - x_{k_k}\| <
\frac{1}{k}$ for any $k\in\N$, and
$(\widetilde{\Phi}(x_{k_k}))_{k\in\N}$ converges to $z$ in the
Euclidean topology of $V$. This implies that the sequence
$(\Phi(x_{k_k}))_{k\in\N}=(\Psi\circ\widetilde{\Phi}(x_{k_k}))_{k\in\N}$
converges to $y=\phi(z)$ in the Euclidean topology of $\mathcal{N}$.
Since $\Phi$ satisfies the condition of Lemma \ref{lemma module} we
conclude that the sequence $(x_{k_k})_{k\in\N}$ is bounded. From the
inequality $\| v_k - x_{k_k}\| < \frac{1}{k}$ for every $k\in\N$ we
infer now that the sequence $(v_k)_{k\in\N}$ is also bounded,
contrary to our assumption that $(v_k)_{k\in\N}$ is unbounded.
\end{proof}

From Claims \ref{claim *} and \ref{claim **} we deduce that the
fibre $\widetilde{\Phi}^{-1}(z)$ is a zero--dimensional affine
variety and that there exists a prime ideal $\mathfrak{p}$ of
$\C[\ol{\mathcal{M}}]$ with $\mathfrak{p} \cap
\C[V]=\mathfrak{M}_z$, which is maximal and minimal under this
condition. Considering $\C[\ol{\mathcal{M}}]$ as a $\C[V]$--module
and taking into account that $\C[V]$ is integrally closed in
$\C[\ol{\mathcal{M}}]$ we may now deduce from Zariski's Main Theorem
(see, e.g., \cite[IV.2]{Iversen73}) that
$\C[V]_{\mathfrak{M}_z}=\C[\ol{\mathcal{M}}]_{\mathfrak{M}_z}$
holds. Since $z$ was an arbitrary point of $V$ with $\Psi(z)=y$ we
conclude
$\C[V]_{\mathfrak{m}_y}=\C[\ol{\mathcal{M}}]_{\mathfrak{m}_y}$.
Finally $\C[\ol{\mathcal{M}}]_{\mathfrak{m}_y}$ is a finite
$\C[\ol{\mathcal{N}}]_{\mathfrak{m}_y}$--module, because $\C[V]$ is
finite over $\C[\ol{\mathcal{N}}]$.
\end{proof}

We finally state the following result, which is used in Appendix
\ref{sec: approximative representations} below.
\begin{lemma}
\label{lemma Phi} Let $K \subset \mathcal{M}$ and $\mathcal{N}$
constructible subsets of suitable affine spaces over $\C$ and let
$\Phi:\mathcal{M}\rightarrow \mathcal{N}$ be a geometrically robust
constructible map. Then the following assertions hold.
\begin{itemize}
    \item[$(i)$] $K$ is closed in the Zariski topology of $\mathcal{M}$
    if and only if it is closed in the Euclidean topology of $\mathcal{M}$.
    \item[$(ii)$] $\Phi$ is continuous with respect to the Zariski
    topologies of $\mathcal{M}$ and $\mathcal{N}$.
    \item[$(iii)$] If $\mathcal{M}$ is irreducible, then the
    constructible set $\Phi(\mathcal{M})$ is also irreducible.
\end{itemize}
\end{lemma}
\begin{proof}
$(iii)$ is an immediate consequence of $(ii)$ and $(ii)$ follows
from $(i)$ and Theorem \ref{theo continuos}. The only if part of
$(i)$ is trivial. Let us show the if part.

Suppose that the constructible set $K$ is closed in the Euclidean
topology of $\mathcal{M}$. Then the Euclidean closure $\ol{K}$ of
$K$ satisfies the condition $\ol{K} \cap \mathcal{M} = K$. Since
$\ol{K}$ is closed with respect to the Zariski topology we see that
$K$ is also Zariski closed.
\end{proof}
%
%
\section{Approximative representations and computational models}
\label{sec: approximative representations}
Complexity models dealing with objects of approximative nature have
been used in computer algebra for several purposes. For example, the
concept of border rank has been one of the keys to the fastest known
matrix multiplication algorithms (see, e.g., \cite[Chapter
15]{Burgisser97}). The concepts of approximative complexity has also
been applied to arbitrary polynomials and rational functions (see,
e.g., \cite{Ald84}, \cite{Griesser86} or \cite{Lic90}).

In this appendix we show that the approximative quiz game protocol
of Section \ref{subsec: approximative game} models the standard
concept of approximative complexity.
%
For this purpose, we introduce a representation of polynomials by
means of approximative information which is directly inspired in
that of the references above, namely by a certain meromorphic map
germ. Then we prove that a polynomial can be represented by an
approximative parameter instance in this sense if and only if it can
be represented in a model closed to that of Section \ref{subsec:
approximative game}, namely by a (not necessarily convergent)
sequence of parameter instances.

Let $\mathcal{O}$ be a framed abstract data type carrier,
$\mathcal{M}$ and $\mathcal{N}$ framed data structures, where
$\mathcal{M}$ is irreducible,
$\mathcal{\theta}:\mathcal{M}\rightarrow\mathcal{O}$ a surjective
geometrically robust constructible map,
$\mu:\mathcal{M}\rightarrow\mathcal{N}$ a geometrically robust
constructible map and $\omega:\mathcal{N}\rightarrow\mathcal{O}$ a
polynomial map such that $\theta=\omega\circ\mu$ holds. Thus
$\theta$ is an abstraction function associated with $\mu$ and
$\omega$. Suppose that the elements of $\mathcal{O}$ are polynomials
of $\C[X_1,\dots,X_n]$, where $X_1\klk X_n$ are indeterminates.

Let $U_1,\dots,U_r$ be new indeterminates, where $r$ is the size of
$\mathcal{M}$, and let $U:=(U_1,\dots,U_r)$ and
$X:=(X_1,\dots,X_n)$. Let $\mathfrak{a}$ be the vanishing ideal of
$\ol{\mathcal{M}}$ in $\C[U]$ and let us fix a polynomial
$P\in\C[U]$ such that $\ol{\mathcal{M}}_{P}:=
\{u\in\ol{\mathcal{M}}:P(u)\neq 0 \}$ is a Zariski open and dense
subset of $\mathcal{M}$, and such that $\theta$ as a rational
function is everywhere defined on $\overline{\mathcal{M}}_P$ (see
Remark \ref{remark piecewise rational}). Let $\epsilon$ be a new
indeterminate.

\begin{definition}
An approximative parameter instance for $\theta$ is a vector
$u(\epsilon)= (u_1(\epsilon),\dots,u_r(\epsilon)) \in
\C{(\!(\epsilon)\!)}^r$ which constitutes a meromorphic map germ at
the origin such that any polynomial of $\mathfrak{a}$ vanishes at
$u(\epsilon)$ and $P(u(\epsilon))\neq 0$ holds.
\end{definition}

Let $u(\epsilon)$ be an approximative parameter instance for
$\theta$. Then there exists an open disc $\Delta$ around $0$ such
that for any complex number $c\in \Delta\setminus \{ 0 \}$ the germ
$u(\epsilon)$ is holomorphic at $c$ and such that $P(u(c))\neq 0$
holds. This implies that any polynomial of $\mathfrak{a}$ vanishes
at $u(c)$. In particular, $u(c)$ belongs to $\mathcal{M}$ for any
$c\in\Delta\setminus\{0\}$.

For technical reasons we need the following result.
\begin{lemma}
\label{lemma open disc} Let $u(\epsilon)$ be an approximative
parameter instance for $\theta$. Then there exists an open disc
$\Delta$ of $\C$ around the origin and a germ $\psi$ of meromorphic
functions at the origin such that $u(\epsilon)$ and $\psi$ are
holomorphic on $\Delta \setminus \{ 0 \}$ and such that any complex
number $c\in\Delta \setminus \{ 0 \}$ satisfies the conditions
$P(u(c))\neq 0$ and $\psi(c)=\mu(u(c))$.
\end{lemma}

\begin{proof}
There exists an open disc $\Delta'$ of $\C$ around the origin such
that $u(\epsilon)$ is everywhere defined on $\Delta' \setminus \{ 0
\}$ and such that any $c \in \Delta' \setminus \{ 0 \}$ satisfies
the condition $P(u(c))\neq 0$. Let $\mathcal{N}_0$ be the Zariski
closure of the image of $\Delta' \setminus \{ 0 \}$ under
$u(\epsilon)$.

\begin{claim}${\mathcal N}_0$ is irreducible.
\end{claim}
\begin{proof}[Proof of the claim]
Let $W_1, \ldots, W_l$ be the irreducible components of ${\mathcal
N}_0$ and let $C_1,\ldots, C_l$ be the pre-images of $W_1,\ldots,
W_l$ under the restriction of $u(\varepsilon)$ to
$\Delta'\setminus\{0\}$. Then $C_1,\ldots, C_l$ are analytic subsets
of the (connected) complex domain $\Delta'\setminus\{ 0\}$ which
satisfy the condition
\begin{equation}\label{irreducibilidad:eqn}
\Delta'\setminus\{0\}=C_1\cup \cdots \cup C_l.
\end{equation}
Each $C_j$ is defined by means of a finite number of equalities and
inequalities of holomorphic functions on $\Delta'\setminus\{0\}$.
The Identity Theorem for holomorphic functions implies that either
each point of $C_j$ is isolated or there exists a nonempty open
subset of $\Delta'\setminus\{0\}$ contained in $C_j$. Since
$\Delta'\setminus\{0\}$ is an open set and equality
(\ref{irreducibilidad:eqn}) holds, there exists an index $j$, $1\leq
j \leq l$, which satisfies the latter condition. Because
$\Delta'\setminus\{0\}$ is connected, the Identity Theorem shows
that $C_j=\Delta'\setminus\{0\}$. Therefore,
$u(C_j)=u(u^{-1}(W_j))\subseteq W_j$ and, thus, $W_j$ contains the
image of $\Delta'\setminus\{0\}$ under $u(\varepsilon)$. This
implies $W_j={\mathcal N}_0$.
\end{proof}

Since ${\mathcal N}_0$ is irreducible, by Remark \ref{remark
piecewise rational} there exists  a Zariski open and dense subset
$\mathcal{U}$ of $\mathcal{N}_0$ with
$\mathcal{U}\subset\mathcal{M}$ such that $\mu$ is rational and
everywhere defined on $\mathcal{U}$. Moreover there exists a
non--zero polynomial $Q\in\C[U]$ such that $\mathcal{N}_Q$ is
contained in $\mathcal{U}$ and Zariski dense in $\mathcal{N}_0$.
Therefore there exists a complex number $c_0\in \Delta'\setminus\{ 0
\}$ with $Q(u(c_0))\neq 0$. Replacing $\Delta'$ by a smaller open
disc whose closure is contained in $\Delta'$ we may use the same
arguments as in the proof of the claim above to show that without
loss of generality $K:=\{ c\in \Delta'\setminus\{ 0 \}: Q(u(c))=0
\}$ is finite. Thus we may choose an open disc $\Delta$ around the
origin with $\Delta \subset \Delta'$ and $\Delta \cap K =\emptyset$
such that the image of $\Delta$ under $u(\epsilon)$ is contained in
$\mathcal{N}_Q$. We conclude now that $u(\epsilon)$ is everywhere
defined on $\Delta\setminus\{ 0 \}$ and that every $c\in\Delta
\setminus \{ 0 \}$ satisfies the condition $u(c)\in\mathcal{U}$. For
$c\in\Delta\setminus\{ 0 \}$, let $\psi(c):=\mu(u(c))$. Then
$\psi:\Delta\setminus\{ 0 \} \to \C^m$ is a well--defined
meromorphic function. Let $c\in\Delta\setminus\{ 0 \}$ and let
$(c_k)_{k\in\N}$ be a sequence in $\Delta \setminus\{ 0 \}$ which
converges to $c$. Then the sequence $(u(c_k))_{k\in\N}$ of points of
$\mathcal{U}$ converges to $u(c)$ and hence the sequence
$(\mu(u(c_k)))_{k\in\N}$ converges to $\mu(u(c))$ and is therefore
bounded. This implies that $\psi$ is holomorphic at $c$. Thus $\psi$
is holomorphic on $\Delta \setminus\{ 0 \}$ and for any $c\in \Delta
\setminus\{ 0 \}$ we have $\psi(c)=\mu(u(c))$ and $P(u(c))\neq 0$.
Therefore we may interpret $\psi$ as a meromorphic map germ at the
origin which is holomorphic on $\Delta \setminus \{ 0 \}$.
\end{proof}

Let $u(\epsilon)$ be an approximative parameter instance for
$\theta$. Then following Lemma \ref{lemma open disc} there exists an
open disc $\Delta$ of $\C$ around the origin such that
$\mu(u(\epsilon))$ is meromorphic on $\Delta$ and holomorphic on
$\Delta \setminus \{ 0 \}$. Therefore the coefficients of the
polynomial $\theta(u(\epsilon)):=\omega(\mu(u(\epsilon)))$ with
respect to $X$ have the same property. Thus $\theta(u(\epsilon))$
can be interpreted as an element of $\C(\!(\epsilon)\!)[X]$.

We say that the approximative parameter instance $u(\epsilon)$ for
$\theta$ \emph{encodes} a polynomial $H\in\C[X]$ if there exists a
polynomial $H'\in\C[\![\epsilon]\!][X]$, whose coefficient vector
with respect to $X$ constitutes a germ of functions which are
holomorphic at the origin, such that $\theta(u(\epsilon))$ can be
written in $\C(\!(\epsilon)\!)[X]$ as
\[
\theta(u(\epsilon))=H + \epsilon H'.
\]
The mere existence of an encoding of a given polynomial by an
approximative parameter instance for $\theta$ becomes characterized
as follows.

\begin{theorem}
\label{aproximative theorem} Let notations and assumptions be as
before and let $H\in\C[X]$. Then the following conditions are
equivalent:
\begin{itemize}
\item[$(i)$] There exists an approximative parameter instance
for $\theta$ that encodes the polynomial $H$.
\item[$(ii)$] There exists a sequence $(u_k)_{k\in\N}$ in $\mathcal{M}$
such that $(\theta(u_k))_{k\in\N}$ converges to $H$ in $\C[X]$.
\item[$(iii)$] $H$ belongs to $\ol{\mathcal{O}}$.
\end{itemize}
\end{theorem}
\begin{proof}
The conditions $(ii)$ and $(iii)$ are obviously equivalent because
one is only a restatement of the other. It suffices therefore to
show the implications $(i)\Rightarrow (ii)$ and
$(ii)+(iii)\Rightarrow (i)$. We first prove $(i)\Rightarrow (ii)$.

Suppose that there exists an approximative parameter instance
$u(\epsilon)$ for $\theta$ such $u(\epsilon)$ encodes $H\in\C[X]$ by
means of a polynomial $H'\in\C[\![\epsilon]\!][X]$ whose coefficient
vector constitutes a holomorphic map germ at the origin. Thus we
have $\theta(u(\epsilon))= H + \epsilon H'$ in
$\C(\!(\epsilon)\!)[X]$. We may choose a sequence
$(\epsilon_k)_{k\in\N}$ of non--zero complex numbers converging to
zero such that for any $k\in\N$ the germ $u(\epsilon)$ is
holomorphic at $\epsilon_k$ and satisfies the condition
$P(u(\epsilon_k))\neq 0$. Without loss of generality we may suppose
that the coefficients of $H'$ are holomorphic at $\epsilon_k$ for
any $k\in\N$.

For $k\in\N$ let $u_k:=u(\epsilon_k)$. Then $u_k$ belongs to
$\mathcal{M}$ and $(\theta(u_k))_{k\in\N}=(H(X)+\epsilon_k
H'(\epsilon_k , X))_{k\in\N}$ converges to $H$. Therefore condition
$(ii)$ is satisfied.

The remaining implication is more cumbersome. In the proof below we
adapt the argumentation of \cite[Lemmas 1 and 2]{Ald84} to our
context.

By assumption $\ol{\mathcal{M}}$ is an irreducible affine variety.
Observe that $\theta(\ol{\mathcal{M}}_{P})$ is Zariski dense in
$\ol{\mathcal{O}}$, because $\theta$ is continuous with respect to
the Euclidean and the Zariski topologies of $\mathcal{M}$ and
$\mathcal{O}$ (see Theorem \ref{theo continuos} and Lemma \ref{lemma
Phi} $(ii)$). Thus $B:=\ol{\ol{\mathcal{O}}
\setminus\theta(\ol{\mathcal{M}}_P)}$ is a proper Zariski closed
subset of $\ol{\mathcal{O}}$. Let $q$ be the dimension of the
irreducible affine variety $\ol{\mathcal{O}}$. By assumption we have
$H\in\ol{\mathcal{O}}$. If $q=0$, then $H\in \mathcal{O}$ and we are
done. Therefore we may suppose without loss of generality $q>0$.

By Noether's Normalization Lemma there exists a surjective finite
morphism of irreducible affine varieties
$\lambda:\ol{\mathcal{O}}\to\A^q$. Since $B$ is a proper Zariski
closed subset of $\ol{\mathcal{O}}$ we have $\lambda(B)\subsetneqq
\A^q$. We may therefore choose a point $z\in\A^q \setminus
\lambda(B)$. Let $L$ be a straight line of $\A^q$ which passes
through $\lambda(H)$ and $z$. Then $\lambda(H)$ belongs to $L$ and
$\lambda(B)\cap L$ is a finite set. Since the morphism $\lambda$ is
finite, the irreducible components of $\lambda^{-1}(L)$ are all
closed curves of $\ol{\mathcal{O}}$ which become mapped onto $L$ by
$\lambda$. In particular, there exists an irreducible component $C$
of $\lambda^{-1}(L)$ which contains $H$. Since $\lambda(B)\cap L$ is
finite, we have $C\nsubseteq B$. Therefore $C\cap B$ is also finite.
Suppose that $\theta(\ol{\mathcal{M}}_{P}) \cap C$ is finite. Then
we may conclude that $C \cap B$ is infinite, a contradiction.
Therefore $\theta(\ol{\mathcal{M}}_P) \cap C$ is infinite. Hence the
Zariski closure of $\theta^{-1}(C)$ in $\ol{\mathcal{M}}$ contains
an irreducible component $V$ such that the constructible set
$\theta(V_P)$ is Zariski dense in $C$. Let $q^*:=\dim V$ and $u \in
V_P$. Observe that $q^* >0$ and, by the continuity of $\theta$,
$B^*:=\ol{\theta^{-1}(\theta(u))} \cap V$ is a proper Zariski closed
subset of $V$.

Again by Noether's Normalization Lemma there exists a surjective
finite morphism of irreducible affine varieties $\lambda^*:V \to
\A^{q^*}$. Since $B^*$ is a proper Zariski closed subset of $V$ we
have $\lambda^*(B^*) \subsetneqq \A^{q^*}$. Therefore we may choose
again a point $z^* \in\A^{q^*} \setminus \lambda^*(B^*)$ and a
straight line $L^*$ of $\A^{q^*}$ which passes through
$\lambda^*(u)$ and $z^*$. Thus $\lambda^*(B^*) \cap L^*$ is a finite
set and $\lambda^*(u)$ belongs to $L^*$. The irreducible components
of $(\lambda^*)^{-1}(L^*)$ are all closed curves of $V$ which become
mapped onto $L^*$ by $\lambda^*$. In particular, there exists an
irreducible component $C^*$ of $(\lambda^*)^{-1}(L^*)$ which
contains $u$. Since $\lambda^*(B^*) \cap L^*$ is finite we conclude
$C^* \nsubseteq B^*$. Moreover, as $C^*$ is irreducible, $u\in C^*$
and $u\in V_P$, we infer that $C^*_P$ is Zariski dense in $C^*$.
Hence $C^* \nsubseteq B^*$ implies that there exists a point $u^*\in
C^*_P \setminus B^*$. By definition of $B^*$, $\theta(u^*) \neq
\theta(u)$. Moreover, we deduce from Lemma \ref{lemma Phi}$(iii)$
that the constructible set $\theta(C^*_P)$ is irreducible. Therefore
$\theta(C^*_P)$ is Zariski dense in $C$.

In this way we have found two irreducible closed curves
$C^*\subseteq\ol{\mathcal{M}}$ and $C\subseteq\ol{\mathcal{O}}$ with
$C^*_P$ nonempty such that $\theta$ maps $C^*_P$ into $C$ and
$\theta(C^*_P)$ is Zariski dense in $C$. Moreover, $H\in C$. The
restriction of $\theta$ to $C^*_P$ is, by assumption on $P$, a
well--defined rational map with Zariski dense image in $C$.
Therefore, the geometrically robust constructible map $\theta$
induces a finite field extension $\C(C) \subset \C(C^*)$.

Let $D^*$ be the normalization of the projective closure of $C^*$
and let $D$ be the projective closure of $C$. Then $\theta$ induces
a rational map $\theta^*:D^*\dashrightarrow D$ whose image is dense
in $D$. Since $\theta^*$ is a smooth curve and $D$ is projective,
$\theta^*$ is a regular map (see, e.g., \cite[\S II.3, Corollary
1]{Shafa94}). The situation is depicted in the following diagram:
\begin{center}
\begin{tikzpicture}
 \node[shape=rectangle] (v1) at (0,2) {$C^*$};
 \node[shape=rectangle] (v2) at (0,0) {$D^*$};
 \node[shape=rectangle] (v3) at (2,2) {$C$};
 \node[shape=rectangle] (v4) at (2,0) {$D$};
 \path[color=black]
 (v1) edge[dashed,->] (v2)
 (v1) edge[dashed,->] node[above] {$\theta$} (v3)
 (v3) edge[->] (v4)
 (v2) edge[->] node[above] {$\theta^*$} (v4);
\end{tikzpicture}
\end{center}

As $\mathrm{im}(\theta^*)$ is dense in $D$, we conclude that
$\theta^*$ is surjective. Hence there exists a point $\eta$ of $D^*$
with $\theta^*(\eta)=H$.
Let 
$\mathcal{S}:=\mathcal{O}_{D^*,\eta}$ be the local ring of $D^*$ at
$\eta$. Thus $\mathcal{S}$ is a regular $\C$--algebra of dimension
one and therefore there exists an embedding of $\mathcal{S}$ into a
power series ring $\C[\![\epsilon]\!]$ which maps any generator of
the maximal ideal of $\mathcal{S}$ onto a power series of order one.
Moreover, by the Jacobian criterion and the Implicit Function
Theorem, the elements of $\mathcal{S}$ become mapped onto power
series which constitute holomorphic function germs at the origin.
Hence the coordinate functions of $C^*$ determined by the
restrictions of the canonical projections $\A^r\rightarrow \A^{1}$
to $C^*$ can be represented by Laurent series
$u_1(\epsilon),\dots,u_r(\epsilon)$ of $\C(\!(\epsilon)\!)$ which
constitute meromorphic function germs at the origin. Let
$u(\epsilon):=(u_1(\epsilon),\dots,u_r(\epsilon))$. Then we deduce
$P(u(\epsilon))\neq 0$ from the fact that $C^*_{P}$ is Zariski dense
in $C^*$ and $\theta(u(\epsilon))$ is a well--defined meromorphic
map germ at the origin by the same argument as in the proof of Lemma
\ref{lemma open disc}.

Furthermore, we have $\theta^*=\theta(u(\epsilon))$, which implies
that $\theta(u(\epsilon))$ admits a holomorphic extension to
$\epsilon=0$. In particular, the entries of the vector
$\theta(u(\epsilon))$ are power series of $\C[\![\epsilon]\!]$ which
constitute holomorphic function germs at the origin. Moreover $H -
\theta(u(\epsilon))$ belongs to the maximal ideal of the local ring
of $\C[D]$ at $H$ and hence to that of $\mathcal{S}$. This means
that $\epsilon$ divides the entries of the coefficient vector of $H
-\theta(u(\epsilon))$ in $\C[\![\epsilon]\!]$. We conclude now that
there exists a polynomial $H'\in\C[\![\epsilon]\!][X]$, whose
coefficients constitute holomorphic function germs at the origin,
such that the equality $\theta(u(\epsilon))=H + \epsilon H'$ holds
in $\C(\!(\epsilon)\!)[X]$. Since $C^*$ is contained in
$\ol{\mathcal{M}}$ we have finally $A(u(\epsilon))=0$ for any
polynomial $A\in\mathfrak{a}$. Thus $u(\epsilon)$ is an
approximative parameter instance for $\theta$ that encodes the
polynomial $H$.
\end{proof}

Theorem \ref{aproximative theorem} suggests that we may consider a
(not necessarily convergent) sequence $(u_k)_{k\in\N}$ in
$\mathcal{M}$ such that $(\theta(u_k))_{k\in\N}$ converges in the
Euclidean topology to a polynomial $H\in\C[X_1,\dots,X_n]$ as an
\emph{approximative encoding} of $H$ with respect to $\theta$. This
motivates the approximative quiz game of Section \ref{subsec:
approximative game}.

A symbolic variant of Theorem \ref{aproximative theorem} for the
representation of polynomials by robust arithmetic circuits with
parameter domain $\A^r$ is the main technical contribution of
\cite{Ald84} (see also \cite[\S A]{Lic90}).
%
%

\newcommand{\etalchar}[1]{$^{#1}$}
\providecommand{\bysame}{\leavevmode\hbox
to3em{\hrulefill}\thinspace}
\providecommand{\MR}{\relax\ifhmode\unskip\space\fi MR }
\providecommand{\MRhref}[2]{%
  \href{http://www.ams.org/mathscinet-getitem?mr=#1}{#2}
} \providecommand{\href}[2]{#2}

\end{document}